\documentclass[a4paper,10pt]{amsart}

\usepackage{lmodern}
\usepackage{doi}

\usepackage[T1]{fontenc}
\usepackage[utf8]{inputenc}
\usepackage[english]{babel}

\usepackage{amssymb,amsfonts,amsthm,amsmath,latexsym}
\usepackage{enumitem, caption, array, esint, bm, graphicx}
\usepackage[usenames]{xcolor}
\usepackage{hyperref}

\usepackage[margin=1.1in]{geometry}

\theoremstyle{plain}
\newtheorem{theorem}{Theorem}[section]
\newtheorem{corollary}[theorem]{Corollary}

\newtheorem{lemma}[theorem]{Lemma}
\newtheorem{proposition}[theorem]{Proposition}

\theoremstyle{remark}
\newtheorem*{remark}{Remark}
\newtheorem*{remarks}{Remarks}

\newcommand\numberthis{\stepcounter{equation}\tag{\theequation}}

\renewcommand{\setminus}{\smallsetminus}
\newcommand{\ssum}[1]{\sum_{\substack{#1}}}

\newcommand{\e}{{\rm e}}
\newcommand{\dd}{{\rm d}}

\newcommand{\ee}{{\varepsilon}}
\newcommand{\dt}{{\tilde \tau}}

\renewcommand{\d}{\tau}
\newcommand{\gsm}{G}

\newcommand{\Conv}{\mathop{\scalebox{1.5}{\raisebox{-0.2ex}{$\ast$}}}}

\newcommand{\B}{{\mathcal B}}

\newcommand{\C}{{\mathbb C}}
\newcommand{\R}{{\mathbb R}}
\renewcommand{\O}{{\mathcal O}}
\newcommand{\Q}{{\mathbb Q}}
\newcommand{\Z}{{\mathbb Z}}
\newcommand{\N}{{\mathbb N}}
\newcommand{\1}{{\mathbf 1}}

\newcommand{\F}{{\mathcal F}}
\newcommand{\cN}{{\mathcal N}}

\newcommand{\bz}{{\mathbf z}}
\newcommand{\bzt}{{\bm \zeta}}
\newcommand{\bbt}{{\bm \beta}}
\newcommand{\bell}{{\bm \ell}}
\newcommand{\bZ}{{\mathbf Z}}
\newcommand{\bb}{{\mathbf b}}

\DeclareMathOperator{\li}{li}

\DeclareMathOperator{\cond}{cond}

\newcommand{\vphi}{{\varphi}}
\renewcommand{\tilde}{\widetilde}
\renewcommand{\bar}{\overline}
\renewcommand{\hat}{\widehat}

\newcommand{\abs}[1]{\left| #1 \right|}

\renewcommand{\mod}[1]{\ ({\rm mod\ }#1)}
\newcommand{\mods}[1]{\ ({\rm mod^*\ }#1)}
\renewcommand\Re{\operatorname{\mathfrak{Re}}}
\renewcommand\Im{\operatorname{\mathfrak{Im}}}
\newcommand{\cu}{{\mathfrak u}}

\numberwithin{equation}{section}

\title[Titchmarsh's problem for multiplicative functions]{Combinatorial identities and Titchmarsh's divisor problem for multiplicative functions}
\date{\today}

\author{Sary Drappeau}
\address{Aix Marseille Université, CNRS, Centrale Marseille, I2M UMR 7373, 13453, Marseille, France}
\email{sary-aurelien.drappeau@univ-amu.fr}

\author{Berke Topacogullari}
\address{EPFL SB MATH TAN, Station 8, 1015 Lausanne, Switzerland}
\email{berke.topacogullari@epfl.ch}

\thanks{We thank O. Ramaré, H. L. Montgomery, R. C. Vaughan, R. de la Bretèche, É. Fouvry and G. Tenenbaum for helpful discussions and remarks on the present work. In particular we thank É. Fouvry and G. Tenenbaum for remarks which led to the second proof presented here. Part of this work was done during a visit of BT to Aix-Marseille university, supported by the French-Austrian joint project MuDeRa (FWF I-1751-N26, ANR-14-CE34-0009).}


\begin{document}

\begin{abstract}
  Given a multiplicative function~$f$ which is periodic over the primes, we obtain a full asymptotic expansion for the shifted convolution sum $\sum_{|h|<n\leq x} f(n) \tau(n-h)$, where~$\tau$ denotes the divisor function and~$ h \in \Z \setminus \{0\} $.
  We consider in particular the special cases where $f$ is the generalized divisor function~$ \tau_z $ with~$z\in\C$, and the characteristic function of sums of two squares (or more generally, ideal norms of abelian extensions).
  As another application, we deduce a full asymptotic expansion in the generalized Titchmarsh divisor problem $\sum_{|h| < n \leq x, \, \omega(n) = k} \tau(n - h)$, where $\omega(n)$ counts the number of distinct prime divisors of~$n$, thus extending a result of Fouvry and Bombieri-Friedlander-Iwaniec.
  
  We present two different proofs: The first relies on an effective combinatorial formula of Heath-Brown's type for the divisor function~$\tau_\alpha$ with~$\alpha\in\Q$, and an interpolation argument in the~$z$-variable for weighted mean values of~$\tau_z$.
  The second is based on an identity of Linnik type for~$\d_z$ and the well-factorability of friable numbers.
\end{abstract}

\maketitle

\section{Introduction}

Understanding correlations of arithmetic functions is a fundamental question in analytic number theory.
In an explicit form, the problem can be stated as determining the asymptotic behaviour of the sum
\begin{equation} \label{eq:general shifted convolution}
  \sum_{1<n\leq x} f(n) g(n-1),
\end{equation}
where \( f, g : \N \to \C \) are arithmetic functions of multiplicative nature.
Many important problems in number theory can be rephrased in terms of correlations of arithmetic functions, the twin prime conjecture or the Goldbach conjecture being two famous examples (see e.g.~\cite[Chapter~1]{Ell94}).
Sums of the form~\eqref{eq:general shifted convolution} also come up prominently in the study of growth properties of \(L\)-functions in the critical strip.
In this context, the problem is known as the shifted convolution problem and has a long and rich history (see~\cite{Mic07} for an overview).

In general, determining the precise asymptotic behaviour of the unweighted correlation~\eqref{eq:general shifted convolution} is a difficult task and only very few unconditional results are known in this direction, all of them requiring at least one of the involved functions to be very close -- in the convolution sense -- to the constant function~$\1$, the divisor function~$\d(n)$ or to Fourier coefficients of \( \operatorname{GL}_2 \)-automorphic forms.
Note that when~$f$ and~$g$ are bounded, the \emph{logarithmically} weighted correlation
$$ \sum_{1<n\leq x} \frac{f(n) g(n-1)}n $$
has been the object of a recent breakthrough of Tao~\cite{Tao}. The case of odd-order correlations for bounded~$f, g$ was recently settled in~\cite{TT}.

In the present paper, we focus on the particularly important case~$g(n) = \d(n)$ of the unweighted problem~\eqref{eq:general shifted convolution}, which is at the edge of current techniques.
If the average value of~$f$ is not too small, it was already observed by Vinogradov~\cite{Vinogradov-BV} (in the case of primes; see also~\cite{Rod65,Hal67}) that simple asymptotic equivalences for the sum
\begin{equation} \label{eq:shited convolution sum}
  \sum_{1<n\leq x} f(n) \d(n-1)
\end{equation}
can be obtained from analogues of the Bombieri-Vinogradov and Brun-Titchmarsh inequalities. We refer to~\cite{Green18, GranvilleShao18} for recent works on this topic.

It is a considerably more difficult problem to obtain full asymptotic expansions for~\eqref{eq:shited convolution sum}, say, with an error term of the form $\O(x(\log x)^{-N})$ where $N>0$ is fixed but can be chosen arbitrarily large.
The gap in difficulty is related to the ``$x^{1/2}$''-barrier for primes in arithmetic progressions on average over moduli.
To our knowledge asymptotic expansions are known for only very few specific examples of functions~\( f \) of arithmetic interest:
\begin{itemize}[leftmargin=2em]
  \item the indicator function of primes~\cite{Fou85,BFI86},
  \item the indicator function of integers without large prime factors~\cite{FouvryTenenbaum1990,D-FI},
  \item the $k$-fold divisor functions~$\d_k(n) $, $k\in\N$, \(k \geq 2 \)~\cite{Motohashi1980,Top16,T-dk}.
\end{itemize}
The methods from the last example can also be used to handle the case where \(f\) is given by Fourier coefficients of \( \operatorname{GL}_2 \)-automorphic forms, although this does not seem to be worked out explicitly in the literature.

The purpose of the present paper is to introduce two new methods which lead to an asymptotic expansion for~\eqref{eq:shited convolution sum} for a wide class of multiplicative functions.
Let \( A, D \geq 1 \) be fixed integers.
Define \( \mathcal F_D(A) \) to be the set of all multiplicative functions~\( f : \N \to \C \) which are $D$-periodic over the primes in the sense that
\[ f(p_1) = f(p_2) \quad \text{for any primes \(p_1\)~and~\(p_2\) with \( p_1 \equiv p_2 \bmod D,\)} \]
and which satisfy the growth condition,
\[ \left| f(n) \right| \leq \d_A(n) \quad \text{for all~\( n \in \N \),} \]
where \( \d_A(n) \) denotes the generalized divisor function.
Our main result is the following preliminary asymptotic formula for the sum~\eqref{eq:shited convolution sum} for~\( f \in \mathcal F_D(A) \).

\begin{theorem} \label{th:main shift 1}
  Let \( A, D, N \geq 1 \). For all~\( f \in \mathcal F_D(A) \) and all~\( x \geq 2 \), we have
  \begin{equation} \label{eq: main theorem shift 1}
    \sum_{1< n \leq x } f(n) \d(n - 1) = 2 \sum_{ \substack{ \text{\(\chi\) primitive} \\ \cond(\chi) \mid D } } \sum_{ \substack{ q \leq \sqrt{ x } \\ \cond(\chi) \mid q} } \frac1{\varphi(q)} \sum_{\substack{q^2 \leq n \leq x \\ (n, q) = 1}} f(n) \chi(n) + \O\left(\frac x{ (\log x)^N}\right),
  \end{equation}
  where the implied constant depends only on~\(A\), \(D\) and~\(N\).
\end{theorem}

\begin{remarks}\
  \begin{itemize}[leftmargin=2em]
  \item The main term in~(\ref{eq: main theorem shift 1}) can be evaluated asymptotically by classical methods, for instance the Selberg-Delange method~\cite[Chapter~II.5]{Ten15}. We spell this out in detail in three particular cases below.
  \item We stress that the implied constant is uniform in all~$f\in\F_D(A)$, and depends only on~$A, D$ and~$N$. This feature can be useful in applications (see Section~\ref{subsection: titchmarsh divisor problem}).
  \item On the other hand, our result is badly behaved with respect to~\(D\), partly due to the use of the Siegel-Walfisz theorem.
    The arguments presented here do not seem sufficient to obtain an improvement in this aspect, although this does not affect our applications.
  \item The error term in~(\ref{eq: main theorem shift 1}) corresponds to an application of the Siegel-Walfisz theorem. If the Riemann hypothesis is true for all Dirichlet~$L$-functions, then it can be improved to~$\O(x^{1-\delta})$ for some absolute constant~$\delta>0$.
  \end{itemize}
\end{remarks}

Theorem~\ref{th:main shift 1} may also be interpreted as a result of Bombieri-Vinogradov type ``beyond~$\sqrt{x}$'' for the average of~$f\in\F_D(A)$ in the residue classes of a fixed integer and without absolute values: For~$f\in\F_D(A)$,
$$ \sum_{q\leq \sqrt{x}} \Big(\ssum{1< n \leq x \\ n\equiv 1 \bmod{q} } f(n)  - \frac1{\vphi(q)}\sum_{ \substack{ \text{\(\chi\) primitive} \\ \cond(\chi) \mid (D,q) } } \ssum{1<n \leq x \\ (n, q) = 1} f(n) \chi(n)\Big) = \O_{A,D,N}\left(\frac x{ (\log x)^N}\right). $$
We refer to~\cite{Green18,GranvilleShao18} for recent works related to this point of view.

In many applications correlation sums with more general shifts appear and it is important to have results which are uniform in large ranges of the involved parameters.
Our methods are robust enough to be applied to these cases as well, and Theorem~\ref{th:main shift 1} is in fact the special case~$a=h=1$ of the following more general result.

\begin{theorem}[General shifts] \label{th:main}
  Let \( A, D, N \geq 1 \).
  There exists an absolute constant~\( \delta > 0 \), such that, for all~\( f \in \mathcal F_D(A) \), all~\( x \geq 2 \) and all~\( a, h \in \Z \) satisfying~\( 1 \leq a, |h| \leq x^\delta \), we have
  \begin{equation*} \label{eq: main theorem}
    \sum_{  |h|/a < n \leq x } f(n) \d(an - h) = M_f(x; h, a) + \O\left( \d( (a, h) ) \frac x{ (\log x)^N } \right),
  \end{equation*}
  where \( M_f(x; a, h) \) is given by
  \[ M_f(x; a, h) := 2 \sum_{ \substack{ \text{\(\chi\) primitive} \\ \cond(\chi) \mid D } } \sum_{ \substack{ q \leq \sqrt{ ax } \\ \cond(\chi) \mid \frac q{ (q, h) } } } \frac{ \overline\chi\left( \tfrac h{ (h, q) } \right) }{ \varphi\left( \frac q{ (h, q) } \right) } \sum_{ \substack{ q^2/a \leq n \leq x \\ (an, q) = (h, q) } } f(n) \chi\left( \tfrac{an}{ (an, q) } \right), \]
  and where the implied constant depends only on~\(A\), \(D\) and~\(N\).
\end{theorem}

Unfortunately, the range of uniformity in~\(h\) in Theorem~\ref{th:main} is comparatively short.
This is due to a known uniformity issue of arguments based on exponential sums estimates underlying our bilinear sums estimate (see~\cite[p.~200]{FI-PAP}).
Out of the same reason, the methods used here are not able to address the dual problem
$$ \sum_{ n = 1 }^{N - 1} f(n) \d(N - n) $$
(for which results are available for instance when~$f=\d$ or~$f = \d_3$, see~\cite{Mot94, Top16}).

We mention that results are known for affine correlations whose linear parts are pairwise independent~\cite{Matthiesen-d2,Matthiesen-f}, or when there is an additional, long enough average over the shift~\cite{Mikawa,MRT1,MRT2}. See also~\cite{ABSR,BSF} for a function field analogue in the large~$q$ limit.

Finally, we mention the work of Pitt~\cite{Pit13}.
He considered an analogue of the Titchmarsh divisor problem (see Section~\ref{subsection: titchmarsh divisor problem}) with the divisor function replaced by Fourier coefficients of holomorphic cusp forms.
In many situations, these Fourier coefficients and the divisor function exhibit a similar behaviour, since the latter can also be viewed as the Fourier coefficients of an Eisenstein series (see \textit{e.g.}~\cite[Chapter~3.4]{Iwa02}).
Remarkably, Pitt obtained an estimate with a power saving in the error term unconditionally, something which is not known for the original Titchmarsh divisor problem.
It seems possible that his ideas can be adapted to our setting, and that one might obtain an analogue of Theorem~\ref{th:main} with the divisor function replaced by Fourier coefficients of holomorphic cusp forms and with a power saving in the error term.
We do not pursue this here.

We apply Theorem~\ref{th:main} to three functions~$f$ of particular arithmetic interest:
\begin{enumerate}
  \item the generalized divisor functions~$\d_z(n)$ with~\( z \in \C \),
  \item the indicator function of integers~$n$ which are norms of an integral ideal in an abelian extension,
  \item the indicator function of integers~$n$ with exactly~$k$ different prime factors.
\end{enumerate}

\subsection{Correlations of divisor functions} \label{subsection:correlations of divisor functions}

Our first application is related to the generalized additive divisor problem, which asks for an asymptotic evaluation of
\[ D_{k, \ell}(x, h) := \sum_{ |h| < n \leq x } \d_k(n) \d_\ell(n + h) \]
for integers~\( k, \ell \geq 2 \).
This problem has received a lot of attention, partly motivated by its connection to the \(2k\)-th moment of the Riemann zeta function (see~\cite[Chapter~4]{Ivi91} or~\cite{CK16, NT}).

It is conjectured that for some constant~\( C_{k, \ell}(h) > 0 \),
\[ D_{k, \ell}(x, h) \sim C_{k, \ell}(h) x (\log x)^{k + \ell - 2}, \]
and it is known~\cite{Henriot} that this is the correct order of magnitude.
However, this has been proven only for the cases where either~\( k = 2 \) or~\( \ell = 2 \).
In these cases, the best-known results in the literature are of the form
$$  D_{k, 2}(x, h) = x P_{k, h}(\log x) + \O\left( x^{ \theta_k + \varepsilon} \right) \quad \text{for} \quad h \ll x^{\eta_k}, $$
where~$P_{k,h}$ is a degree~$k$ polynomial depending on~$h$, with
\begin{alignat*}{3}
  \theta_2 &= \tfrac23 \quad &\text{and} \quad \eta_2 &= \tfrac23 \qquad &&\text{\cite{DI82, Mot94},} \\
  \theta_3 &= \tfrac89 \quad &\text{and} \quad \eta_3 &= \tfrac23 \qquad &&\text{\cite{FI-d3, Top16},} \\
  \theta_k &= \max\big( 1 - \tfrac4{15k - 9}, \tfrac{56}{57} \big) \quad &\text{and} \quad \eta_k &= \tfrac{15}{19} \qquad \text{(\( k \geq 4 \)~fixed)} \qquad &&\text{\cite{Linnik, FT-dk, T-dk}.}
\end{alignat*}
In the case~\( k = \ell = 2 \), a similar asymptotic formula holds in a much larger range of uniformity for~\(h\), although with a weaker error term (see~\cite{Meu01} for the currently best results in this direction).
For~\( k, \ell \geq 3 \) the problem remains completely open.

The functions~\( \d_k \) are special cases of coefficients of the Dirichlet series
$$ \sum_{n = 1}^\infty \frac{\d_z(n)}{n^s} := \zeta(s)^z \quad \text{for} \quad z \in \C \quad\text{and}\quad \Re(s)>1. $$
On prime powers, they are given explicitly by
\begin{equation} \label{eq:dalpha-pnu}
  \d_z(p^\ell) = \binom{z + \ell - 1}\ell.
\end{equation}
The functions~$\d_z$ for~$z\not\in\N$ have a more complicated behaviour than those for~$z\in\N$. When~\( z = - 1 \) for instance, we recover the M\"obius function~\( \d_{-1}(n) = \mu(n) \).

Theorem~\ref{eq: main theorem} leads to an asymptotic expansion of~\( D_{z, 2}(x, h) \) for arbitrary~\( z \in \C \), uniformly in any fixed disk~$\abs{z}\ll 1$.

\begin{theorem} \label{cor:piltz}
  Let~\( A, N \geq 1 \) and~\( \varepsilon > 0 \).
  There exist a constant~\( \delta > 0 \) and holomorphic functions~\( \lambda_{h, \ell} : \C \to \C \), such that, for~\( |z| \leq A \), \( x \geq 2 \) and~\( 1 \leq |h| \leq x^\delta \),
  \begin{equation} \label{eq:estim-piltz}
    \sum_{ |h| < n \leq x} \d_z(n) \d(n + h) = x (\log x)^z \sum_{\ell = 0}^N \frac{ \lambda_{h, \ell}(z) }{ (\log x)^\ell } + \O\left( \frac{ x (\log x)^{ \Re(z) } }{ (\log x)^{ N + 1 - \varepsilon } } \right),
  \end{equation}
  where the implicit constant only depends on~\( A\),~\(N\)~and~\(\varepsilon\).
\end{theorem}

The coefficients~\( \lambda_{h, \ell}(z) \) can be computed explicitly; see~\eqref{eq:first-coeff-dz} \textit{infra} for an expression of the leading coefficient.
If \(z\) is a non-positive integer, all the coefficients~\( \lambda_{h, \ell}(z) \) vanish and \eqref{eq:estim-piltz} effectively becomes an upper bound.

Our method leads to a power saving error term in Theorem~\ref{cor:piltz} when~\( z = k\in \N \).
This is solely due to the fact that in these cases the \( k \)-th power of Dirichlet \(L\)-functions~\( L(s, \chi)^k \) can be continued analytically to a strip~$\Re(s)\geq 1-\delta$ for some~$\delta>0$ (excluding the possible pole at~\( s = 1 \)).
We do not focus of the case~$z\in \N$ here, since the works mentioned above then give quantitatively stronger estimates.

\subsection{Norms of integral ideals}

Let~$K/\Q$ be a Galois extension with discriminant~\( \Delta_K \).
We define
$$ \cN_K := \{ N(\alpha) : \alpha \text{ ideal of }\O_K, \, \alpha\neq 0\}. $$
This set has a rich multiplicative structure, described by the Artin reciprocity law.
When the extension is abelian, the Dedekind function~$\zeta_K(s)$ factorizes into Dirichlet $L$-functions mod~$\Delta_K$, so that the integers in~$\cN_K$ can be detected by looking at the congruence classes of their prime factors mod~$\Delta_K$.
Theorem~\ref{th:main} eventually applies and leads to the following result.

\begin{theorem} \label{cor:norms of ideals}
  Let~$K/\Q$ be an abelian field extension.
  Let~\( N \geq 1 \) and~\( \varepsilon > 0 \).
  There exist a constant~\( \delta > 0 \) and real numbers~\( \kappa_{h, \ell}(K) \), such that, for~\( x \geq 2 \) and~\( 1 \leq |h| \leq x^\delta \),
  \begin{equation} \label{eq:estim-frob}
    \sum_{ \substack{ |h| < n \leq x \\ n\in\cN_K} } \d(n - h) = x (\log x )^{1-1/[K:\Q]} \sum_{\ell = 0}^N \frac{ \kappa_{h, \ell}(K) }{ (\log x)^\ell } + \O\left( \frac{ x }{ (\log x)^{N + 1/[K:\Q] - \varepsilon} } \right),
  \end{equation}
  where the implicit constant depends only on~\(K\),~\( N \) and~\(\varepsilon\).
\end{theorem}

An interesting special case is given by the extension~\( \Q(i) / \Q \).
In this case, \( \mathcal N_{ \Q(i) } \) is simply the set of integers which can be written as a sum of two squares, and Theorem~\ref{cor:norms of ideals} takes the following form.

\begin{corollary} \label{cor:sums of two squares}
  Let~\( \mathcal B \) be the set of all integers which can be written as a sum of two squares.
  Let~\( N \geq 1\) and~\( \varepsilon > 0 \).
  There exist a constant~\( \delta > 0 \) and real numbers~\( \beta_{h, \ell} \), such that, for~\( x \geq 2 \) and~\( 1 \leq |h| \leq x^\delta \),
  \begin{equation} \label{eq:estim-bn}
    \sum_{ \substack{ |h| < n \leq x \\ n \in \B } } \d(n - h) = x ( \log x )^\frac12 \sum_{\ell = 0}^N \frac{ \beta_{h, \ell} }{ (\log x)^\ell } + \O\left( \frac{ x }{ (\log x)^{N + 1/2 - \varepsilon} } \right).
  \end{equation}
  where the implicit constant depends only on~\( N \) and~\(\varepsilon\).
\end{corollary}

The first term in the asymptotic formula for the left-hand side of~\eqref{eq:estim-bn} can also be obtained using a recent extension of the Bombieri-Vinogradov theorem due to Granville and Shao~\cite{GranvilleShao18}, along with the Brun-Titchmarsh inequality.
The coefficients~\( \kappa_{h, \ell}(K) \) and~\( \beta_{h, \ell} \) can be computed explicitly; see \eqref{eq:first-coeff-bn} \textit{infra} for an evaluation of the leading coefficient~$\beta_{h,0}$ in~\eqref{eq:estim-bn}.
Note that, since the indicator function~\( b(n) \) of the set~\( \mathcal B \) correlates with both the principal and the non-principal character mod~\(4\), there are two genuine contributions on the right-hand side in~\eqref{eq: main theorem shift 1} when~\( f(n) = b(n) \).
This also explains the discrepancy between the conjectures made in~\cite{Iwa76} and~\cite{FKR} on autocorrelations of~$b(n)$.

We stress that the multiplicity of representations as ideal norms in Theorem~\ref{cor:sums of two squares} is \emph{not} taken into account.
Thus the estimate~\eqref{eq:estim-bn} is more difficult to obtain than an estimate for the correlation sum
\[ \sum_{ |h| < n \leq x } r_2(n) \d(n - h) \quad \text{with} \quad r_2(n) := \abs{\{(r,s)\in\Z^2 : r^2+s^2=n\}}, \]
for which classical methods suffice.

\subsection{Integers with \texorpdfstring{\(k\)}{k}~prime divisors} \label{subsection: titchmarsh divisor problem}

The Titchmarsh divisor problem, posed in 1930~\cite{Tit30}, asks for an asymptotic evaluation of the sum
\begin{equation} \label{eq: titchmarsh sum}
  \sum_{ |h| < p \leq x } \d(p - h),
\end{equation}
where \(p\) runs over all primes up to~\(x\).
Following the initial works by Titchmarsh~\cite{Tit30} and Linnik~\cite{Linnik}, the best known result was obtained independently by Fouvry~\cite{Fou85} and Bombieri, Friedlander and Iwaniec~\cite{BFI86}: For any fixed~\( N > 0 \), we have, for~\( 1 \leq |h| \leq (\log x)^N \),
\begin{equation} \label{eq: best result for titchmarsh}
  \sum_{ |h| < p \leq x } \d(p - h) = C_h x + C_h' \li(x) + \O\left( \frac x{ (\log x)^N } \right),
\end{equation}
where
\[ C_h = \frac{\zeta(2) \zeta(3)}{\zeta(6)} \prod_{p \mid h} \left( 1 - \frac p{p^2 - p + 1} \right), \quad C_h' = \bigg( \gamma - \sum_p \frac{\log p}{p^2 - p + 1} + \sum_{p \mid h} \frac{ p^2 \log p }{ (p - 1) (p^2 - p + 1) } \bigg) C_h. \]
An interesting generalization of this problem concerns the sum
\begin{equation} \label{eq: titchmarsh for general k}
  \sum_{ \substack{ |h| < n \leq x \\ \omega(n) = k } } \d(n - h),
\end{equation}
where \( \omega(n) \) denotes the number of distinct prime divisors of an integer~\(n\).
An asymptotic equivalence for this sum was proven by Khripunova~\cite[Theorem~3]{Khr98}, uniformly for~\( k \ll \log\log x \) and~\( h \ll x \).

Our methods allow to obtain a full asymptotic expansion for~\eqref{eq: titchmarsh for general k}, at least for small shifts~\(h\).
In order to circumvent the obstacle that the indicator function for integers~\(n\) with~\( \omega(n) = k \) is not multiplicative, we use a classical method due to Selberg~\cite{Sel54}, which allows us to reduce the evaluation of~\eqref{eq: titchmarsh for general k} to the evaluation of the correlation sum of the divisor function with the multiplicative function~\( n \mapsto z^{ \omega(n) } \).
This eventually leads to the following result.

\begin{theorem} \label{cor:titchmarsh}
  Let~\( N \geq 1 \) and~\( \varepsilon > 0 \).
  There exist a constant~\( \delta > 0 \) and polynomials~\( P_{h, \ell}^k(X) \) of degree~\(k - 1\) such that, for~\( 1 \leq k \ll \log\log x \) and~\( |h| \leq x^\delta \),
  \begin{equation} \label{eq: asymptotic expansion for general k}
    \sum_{ \substack{ |h| < n \leq x \\ \omega(n) = k } } \d(n - h) = x \sum_{ \substack{ 0 \leq \ell \leq N } } \frac{ P_{h, \ell}^k(\log\log x) }{ (\log x)^\ell } + \O\left( \frac{ x (\log\log x)^k }{ k! (\log x)^{N + 1 - \varepsilon} } \right),
  \end{equation}
  where the implicit constants depend only on~\(N\) and~\( \varepsilon \).
\end{theorem}

The case~\( k = 1 \) recovers the best-known asymptotic formula~\eqref{eq: best result for titchmarsh} for the Titchmarsh divisor problem.
As before, the polynomials~\( P_{h, \ell}^k \) can be computed explicitly; in particular, the leading coefficient in the asymptotic expansion is given by~\( C_h / (k - 1)! \).

This result is non-trivial throughout the range~\( k \ll \log\log x \).
The case~$k/\log \log x \to +\infty$ is an interesting question which would require different tools, due to the sparsity of the set of integers under consideration (not unlike the situation for friable integers~\cite{HarperBV2012}). We do not address this here.

\subsection{Overview of the proof of Theorem~\ref{th:main}} \label{sec:overview}

For the sake of clear exposition, we will focus here on the case~\( D = 1 \), as our arguments extend without much difficulty to the case of general moduli and the arising complications are mainly of technical nature. Note that any~\( f \in \mathcal F_1(A) \) can be approximated (in the convolution sense) by a suitable generalized divisor function, so that it suffices to consider the case~\( f = \d_z \) with~\( z \in \C \).

We will give two distinct proofs of Theorem~\ref{th:main}.
They are based on two different kinds of combinatorial identities for the generalized divisor function~\( \d_z \), both of which we believe are of independent interest.
Our first approach relies on an effective combinatorial formula of Heath-Brown's type for the divisor function~$\tau_\alpha$ with~$\alpha\in\Q$, and an interpolation argument in the~$z$-variable for weighted mean values of~$\tau_z$.
Our second apprach, which is more direct and avoids the interpolation step, is instead based on an identity of Linnik type for~$\d_z$ and the well-factorability of friable numbers\footnote{The second proof was found only after a preliminary version of the present manuscript was uploaded online.}.

\subsubsection{Proof by Heath-Brown's identity and interpolation}

Our first proof of Theorem~\ref{th:main} divides into two parts:
We first prove the theorem for rational~\(z\), and then extend this result to all~\( z \in \C \).

For~\( z \in \Q \), the general structure of the proof of Theorem~\ref{th:main} follows the setup of~\cite{Fou85, BFI86} (see also~\cite{Fou84}).
The strategy naturally splits into two steps:
\begin{enumerate}
  \item We decompose the function~$f$ into convolutions with either large smooth components (type I) or suitably localized components (type II).
  \item We solve the question for both types of sums.
\end{enumerate}

The bulk of the present work concerns the first step.
Combinatorial decompositions for prime numbers have a long history since the works of Vinogradov~\cite{Vin37} (we refer to the survey~\cite{Ram13} for an account and further references).
Yet, it was not until recently that analogous identities emerged for generalized divisor functions.
Montgomery and Vaughan (private communication) have recently developed a combinatorial identity of Vaughan's type~\cite{Vau75} for~$\d_{1/2}$, which initially motivated largely the present work.
Unfortunately, as for primes, the bilinear sums coming from a raw application of this identity are not quite localized enough to be effective for Titchmarsh's problem, and even though this can sometimes be fixed by iterating the formula~\cite{FouvryDE}, our early attempts were unsuccessful.
Instead we follow the more flexible approach of Heath-Brown~\cite{HB} (which is related to~\cite{Gall68}).

Our first result (Theorem~\ref{th:combinatorial-identity} below) is a uniform combinatorial formula of Heath-Brown's type for the divisor function~$\d_{\frac uv}$ with~$u/v\in\Q$. In the simplest case~$0 < u < v$, it reads
\begin{equation}
\d_\frac uv(n) = \sum_{\ell = 1}^K c_{\ell, K, u/v} \mathop{ \sum \cdots \sum }_{ \substack{ m_1 \cdots m_\ell n_1 \cdots n_{\ell v - u} = n \\ n_1, \ldots, n_{\ell v - u} \leq x^{1/K} } } \d_{-\frac 1v}(n_1) \cdots \d_{-\frac 1v}( n_{\ell v - u} ) \quad \text{for} \quad n \leq x,\label{eq:prototype-combinatorial-formula}
\end{equation}
where~$K\in\N_{>0}$ is arbitrary and where~$c_{\ell,K,u/v} \in \Q$.
A more general formula holds for any rational number~$u/v$ (see Theorem~\ref{th:combinatorial-identity}).
A crucial property of this formula is that it is sensitive almost only to the archimedean size of~$u/v$.
Indeed, for~$\abs{u/v}\leq A$, the coefficients~$c_{\ell,K,u/v}$, the length of the~$\ell$-sum and the value at primes~$n=p$ of each $\ell$-summand on the right-hand side are bounded in terms of~$A$ and~$K$ only (but not of~$v$).
Thus, the only loss due to the size of~$v$ comes from the number~$\O(v)$ of terms in the convolution, which has essentially no effect on what follows.

In the same way, we can express any rational convolution power~$\Conv^{u/v}f$ of a multiplicative function in terms of higher convolutions~$\Conv^k f$ with~$1\leq k \leq K$ and a bilinear term with one component supported on the interval~$[x^\ee, x^{1/K}]$.
However, to our knowledge asymptotic formulae for the correlation sums
\begin{equation} \label{eq:convol-deg-2}
  \sum_{n\leq x} ( {\Conv}^k f )(n) \d(n + 1),
\end{equation}
for~\( k \geq 2 \) are currently known for only very few functions~\(f\) (essentially constant functions and Dirichlet characters).
This is the main obstacle towards using decompositions of this form to prove Theorem~\ref{th:main} for complex-fold convolutions of multiplicative functions.

Regarding the second step, we are mostly able to use the harmonic analysis arguments underlying~\cite{Fou85, BFI86}. They are based on bounds on Kloosterman sums on average~\cite{DI-Kloo}, along with Voronoi summation (for type I) and Linnik's dispersion method (for type II).
We will follow the treatment made in~\cite{D-disp,T-dk}, although some work is needed in order to cast the main terms from these works in a form suitable for us.

Eventually, the arguments described above yield a proof of Theorem~\ref{th:main} for~$f = \d_{\frac uv}$ uniformly in the range~$v \leq (\log x)^N$.
As it turns out, this is already sufficient information to be able to conclude.

To see why, we return to the correlation sum
$$ D(z) := \sum_{|h|/a<n\leq x} \d_z(n) \d(an-h) $$
with~$z\in\C$, $\abs{z}\ll 1$.
The main observation is that this expression is a polynomial in~\( z \), and that we know how to evaluate it on rational numbers with small denominators.
Even though \( D(z) \) initially has degree of the order of~\( \log x \), we can use large deviation bounds on the function~$\omega(n)$ (and a convolution argument) to approximate it, up to an admissible error, by the polynomial 
$$ {\tilde D}(z) := \ssum{|h|/a<n\leq x \\ \omega(n) \ll \log\log x} z^{\omega(n)} \d(an-h), $$
which has degree at most~$\O(\log\log x)$.
This enables us to use Lagrange interpolation on a suitably chosen set of rational sample points to transfer our estimates for~\( z \in \Q \) to estimates of the same quality for~\( z \in \C \).
Indeed, this process introduces an error which grows exponentially in the degree of the polynomial.
As our estimates for~\( D(z) \) for~\( z \in \Q \) save an arbitrarily large power of~\( \log x \), we are still able to obtain an asymptotic formula at the end.

Note that for the above arguments to work it is crucial that estimates with a saving of a large power of~$\log x$ for~$D(z)$ for~\( z \in \Q \) are available, which we can fortunately obtain here from the Siegel-Walfisz bound (an unfortunate consequence of the last fact, however, is that most of our results are not effective).

We mention that, as in Heath-Brown's work~\cite{HB}, the arguments sketched above can be used to obtain asymptotic formulae for short sums
$$ \sum_{x<n\leq x+y} f(n) $$
for~$y \geq x^{7/12 + \ee}$ and~$f\in \F_D(A)$, as well as theorems of Bombieri-Vinogradov type.
However, unlike Titchmarsh's divisor problem, such results could in principle also be obtained by zero-density estimates for Dirichlet $L$-functions (see~\cite[Chapter~10.5]{IK04}, \cite{Bombieri-BV}).

\subsubsection{Proof by Linnik's identity}

Our second proof uses a different decomposition for~\( \d_z \), which has the major advantage that it holds uniformly for all \( z \) in a fixed bounded subset of~\( \C \).
This avoids the interpolation step necessairy in the first proof, although the resulting combinatorial identity is not as elegant as the identity of Heath-Brown's type described above.

A naive attempt to find a combinatorial formula for~$\d_z$ which is uniform in~$z$ might start with Linnik's formula~\cite[\S 13.3]{IK04}, which relies on the Taylor series expansion
$$ \zeta(s)^z = (1+(\zeta(s)-1))^z = \sum_{j\geq 0} \binom{z}{j} (\zeta(s)-1)^j. $$
The main technical difficulty at this point is to truncate the sum over~$j$. In the context of Linnik's formula, this truncation is performed by restricting to almost-primes from the outset (or inserting a sieve weight), see~\cite[p.21]{Linnik}, but unfortunately this approach is not available in our situation.

Instead we write \( \zeta(s) = \zeta_y(s) M_y(s) \), where
\[ \zeta_y(s) := \prod_{ p \leq y } \left( 1 - \frac1{p^s} \right)^{-1} \quad\text{and} \quad M_y(s) := \frac{ \zeta(s) }{ \zeta_y(s) }, \]
with~\( y = x^{1/K} \) for some~\( K \in \N \), and then apply the Taylor series expansion only on the second factor~\( M_y(s) \), so that
$$ \zeta(s)^z = \zeta_y(s)^z \sum_{j\geq 0} \binom{z}{j} (M_y(s)-1)^j. $$
This expression has the advantage that the \(j\)-th summand has no coefficient for~\( n \leq y^j \) in its Dirichlet series expansion.
After expanding and comparing the Dirichlet coefficients on both sides, we are therefore led to the following ``raw'' combinatorial decomposition (see Theorem~\ref{th:combinatorial-identity-linnik}),
\[ \tau_z(n) = \sum_{ 0 \leq \ell < K } c_\ell \sum_{ \substack{ n = n_1 n_2 \\ \text{\(n_1\) is \(y\)-friable} } } \tau_{z - \ell}(n_1) \tau_\ell(n_2) \quad \text{for} \quad n \leq x, \]
where the \(c_\ell\) are some complex numbers which depend on~\(z\), but which can be bound uniformly for~\( z \ll 1 \) (we recall that an integer is said to be \(y\)-friable if all of its prime factors are bounded by~\(y\)).

In order to apply this formula, it is of course necessairy to be able to control the factors~\( \tau_{z - \ell}(n_1) \).
However, the characteristic function of~$y$-friable numbers has good factorability properties (see~\cite[p.66]{Vaughan1989} or~\cite[Lemme 3.1]{FouvryTenenbaum1996}): we can essentially replace them in the formula above by convolutions of sequences supported on~\( [1, y] \) (see Lemma~\ref{lem:friable-well-fact}).
This in turn enables us to apply estimates of type~I and type~II, leading eventually to the desired asymptotic formula.

\subsection*{Plan}

In Section~\ref{section: reduction to F' and F''}, we introduce our main notations and the subsets of functions of~$\F_D(A)$ we will mainly work with.
In Section~\ref{section: combinatorial identities}, we present the combinatorial decompositions for~$\d_z$, on which our proofs are based.
In Section~\ref{section: auxiliary estimates}, we state some auxiliary computations in order to use the results of~\cite{T-dk,D-disp}.
In Sections~\ref{section:The case of rational parameters} and~\ref{section: interpolation to complex parameters}, we proof Theorem~\ref{th:main} using the combinatorial identity of Heath-Brown's type, first by treating the case of rational parameters, and then by interpolating the obtained results to all functions in~$\F_D(A)$.
In Section~\ref{section: proof using Linnik's identity}, we sketch an alternative proof using the combinatorial identity of Linnik's type.
Finally, in Section~\ref{section:proof of main theorem}, we estimate the main terms and prove Theorems~\ref{cor:piltz}, \ref{cor:norms of ideals} and~\ref{cor:titchmarsh}.

\section{First reductions} \label{section: reduction to F' and F''}

\subsection{Statement of the main proposition}

For $n, h \in \Z$ with~$n\geq 1$ and~$n-h\geq 1$, let
\begin{equation} \label{eq:def-dthR}
  \dt_h(n; R) := 2\ssum{q\leq \sqrt{n-h} \\ (n, q) = (h, q)} \frac1{\vphi(\tfrac q{(h, q)})} \ssum{\chi\mod{q/(h, q)} \\ \cond(\chi) \leq R} \chi(\tfrac h{(h, q)}) \overline{\chi(\tfrac{n}{(h, q)})}.
\end{equation}
Note that $\dt_h(n;R) = \d(n-h)$ if~$R>\sqrt{n-h}$ and $n-h$ is not a perfect square.
We will eventually choose~$R$ of size~$(\log n)^{O(1)}$. We have a trivial bound
\begin{equation}
\dt_h(n; R) \ll_\ee  n^\ee R^{1+\ee}.
\end{equation}

The function~\( \dt_h(n; R) \) should be thought of as an approximation to~\( \d(n - h) \) on average.
The main work in proving Theorem~\ref{th:main} consists in showing that, for any~\( f \in \mathcal F_D(A) \), we have
\begin{equation} \label{eq: main result in rought terms}
  \sum_{n \leq x} f(n) \d(an - h) \sim \sum_{n \leq x} f(n) \dt_h( an; R_x ) \quad \text{for} \quad x \to \infty,
\end{equation}
where \(R_x\) is some slowly growing function in~\(x\) (some appropriate power of~\( \log x \)).
Once this is established, we can evaluate the sum on the right by standard methods.
In view of this, it is convenient to define
\[ \Delta_h(n;R) := \d(n-h) - \dt_h(n;R) \quad \text{and} \quad \Sigma_f(I; a, h; R) := \sum_{n \in I} f(n) \Delta_h(an; R), \]
for any interval \(I \subseteq \R^+ \).
The main part of this article is concerned with proving the following proposition, which puts the statement~\eqref{eq: main result in rought terms} into precise terms, and from which the results described in the introduction can be deduced easily (see Section~\ref{section:proof of main theorem}).

\begin{proposition} \label{prop:bound-largechars}
  Let \( A, D \geq 1 \) be fixed.
  Then we have, for~\( x \geq 3 \), \( I \subset [ x/2, x ] \) an interval and \( f \in \mathcal F_D(A) \), the following estimate,
  \begin{equation} \label{eq:main bound}
    | \Sigma_f(I; a, h; R) | \leq C \d((a,h)) \frac{ x (\log x)^B }{ R^{1/3} } \quad \text{for} \quad 1 \leq a, |h|, R \leq x^\delta,
  \end{equation}
  where \( \delta > 0 \) is some absolute constant and where \( B, C > 0 \) are constants which depend only on~\(A\) and~\(D\).
\end{proposition}

\subsection{Restricting the set of functions}

It is known in multiplicative number theory that, to a certain degree of precision, the magnitude of the mean value of a multiplicative function~$f$ depends mostly on the values~$f(p)$, $p$ prime. The following lemma quantifies the analogous phenomenon in our case.

\begin{lemma} \label{lemma:replacing f by g}
  Let \( f, g : \N \to \C \) be multiplicative functions, which satisfy the following conditions,
  \begin{enumerate}[label=(\roman{enumi})]
    \item \( \displaystyle | g(n) | \leq \d_M(n) \quad \text{for some }M\geq 1 \text{ and all } n \in \N, \) \label{enum: first condition}
    \item \( \displaystyle H := \sum_{n\geq 1} \frac{\abs{(f\ast g^{-1})(n)}}{n^\sigma} < +\infty \) for some~$\sigma<1$. \label{enum: second condition}
  \end{enumerate}
  Furthermore, assume there are constants~\( \varrho, \delta \in (0, 1] \) and~\( B, C \geq 1 \) such that, for all~\( x \geq 1 \) and all intervals~\( I \subset [ x/2, x ] \),
  \begin{alignat}{2}
    | \Sigma_g(I; a, h; R) | &\leq C \d( (a, h) ) \frac{ x (\log x)^B }{ R^\varrho } \quad &\text{for} \quad 1 \leq a, |h|, R &\leq x^\delta. \label{eq:assumption on g}
    \intertext{Then there exists~\( C', \delta' > 0 \) depending only on~\( \varrho\), \( \delta\), \(\sigma\) and~\(M\), such that, for all~\( x \geq 1 \) and all intervals~\( I \subset [ x/2, x ] \),}
    | \Sigma_f(I; a, h; R) | &\leq H C C' \d( (a, h) ) \frac{ x (\log x)^B }{ R^\varrho } \quad &\text{for} \quad 1 \leq a, |h|, R &\leq x^{\delta'}.
  \end{alignat}
\end{lemma}
\begin{proof}
  Let \( h := f \ast g^{-1} \).
  We have
  \begin{align*}
    \Sigma_f(I; a, h; R) &= \sum_{ n_1 n_2 \in I } g(n_1) h(n_2) \Delta_h(an_1 n_2; R) \\
      &= \sum_{ n_2 \leq T } h(n_2) \Sigma_g( I / n_2; a n_2, h; R ) + \sum_{ n_2 > T } h(n_2) \Sigma_g( I / n_2; a n_2, h; R ),
  \end{align*}
  for some parameter~\( T \geq 1 \). For the sum on the left we use the assumption~\eqref{eq:assumption on g}, so that
  \[ \sum_{ n_2 \leq T } h(n_2) \Sigma_g( I / n_2; a n_2, h; R ) \ll_{\sigma} C H \d( (a, h) ) \frac{ x (\log x)^B }{ R^\varrho }, \]
  provided that the parameters \(a\), \(h\) and~\(R\) satisfy
  \[ 1 \leq a \leq \frac{ x^\delta }{ T^{1 + \delta} } \quad \text{and} \quad 1 \leq |h|, R \leq \frac{ x^\delta }{ T^\delta }. \]
  For the sum on the right we use the trivial bound~$\Sigma_g( I / n_2; a n_2, h; R) \ll_{\ee, M} R x^{1+\ee}/n_2$, and get
  \[ \sum_{ n_2 > T } h(n_2) \Sigma_g( I / n_2; a n_2, h; R ) \ll_{\ee, M} x^{1+\ee} R T^{-1+\sigma} H. \]
  The lemma follows on setting~\( T = x^{\delta/3}\) and~$\delta' = \min(\frac\delta3, \frac{\delta(1-\sigma)}{4(1+\rho)})$.
\end{proof}

In view of this, in order to prove Proposition~\ref{prop:bound-largechars}, we will restrict to the following two subsets of~\( \mathcal F_D(A) \).
The first subset, denoted by~\( \mathcal F_D^\d(A) \), consists of functions~\( f : \N \to \C \),  which are the coefficents of Dirichlet series of the form
\begin{equation} \label{eq:definition of F'}
  \sum_{n = 1}^\infty \frac{ f(n) }{ n^s } = \prod_{ \chi \bmod D } L(s, \chi)^{b_\chi},
\end{equation}
where the parameters~\( b_\chi \) are complex numbers such that~\( |b_\chi| \leq A \). Note that~$\d_z\in\F_D^\d(A)$ for~$\abs{z} \leq A$.
A particularly important role will be played by the subset~\( \mathcal F_D^{\d_\Q}(A) \subset \mathcal F_D^\d(A) \) formed by functions of this form where all the parameters~\( b_\chi \) are rational.

The second subset~\( \mathcal F_D^\omega(A) \) is defined to be the set of functions~\( f : \N \to \C \), which are the coefficients of Dirichlet series of the form
\begin{equation} \label{eq:definition of F''}
  \sum_{n = 1}^\infty \frac{ f(n) }{n^s} = \prod_{r \in (\Z/D\Z)^\times} \prod_{ p \equiv r \bmod D } \left( 1 + \frac{ z_r }{ p^s - 1 } \right),
\end{equation}
where the coefficients~\( z_r \) are complex numbers such that~\( |z_r| \leq A \). This includes the functions~$n\mapsto z^{\omega(n)}$ for all~$\abs{z}\leq A$.

\begin{lemma} \label{lemma: existence of g_1 and g_2}
  For any~\( f \in \mathcal F_D(A) \), there exist \( g_1 \in \mathcal F_D^\d(A) \)~and~\( g_2 \in \mathcal F_D^\omega(A) \) which satisfy the conditions~\ref{enum: first condition}--\ref{enum: second condition} stated in Lemma~\ref{lemma:replacing f by g} for~$\sigma = \frac23$, and~\(M, H\) bounded only in terms of~\(A\) and~\(D\).
\end{lemma}
\begin{proof}
  We first prove the lemma with respect to the set~\( \mathcal F_D^\d(A) \).
  Let \( f \in \mathcal F_D(A) \) be fixed, and let~\( v_f : \Z \to \C \) be the \(D\)-periodic function defined by
  \begin{equation} \label{eq: definition of v_f}
    v_f(r) = \begin{cases}
               f(p) \quad &\text{if there exists a prime~\(p\) such that \( (p, D) = 1 \) and \( p \equiv r \bmod D\),} \\
               0 \quad &\text{if \( (r, D) > 1 \).}
              \end{cases}
  \end{equation}
  We then set
  \begin{alignat*}{2}
    b_\chi &:= \frac1{\vphi(D)} \sum_{r \mod D} v_f(r) \overline{\chi}(r) \quad &\text{for any character \( \chi\bmod{D} \),}
  \end{alignat*}
  and define \( g_1 \) as the coefficients of the following Dirichlet series,
  \begin{equation} \label{eq: g_1 as a Dirichlet series}
    \sum_{n = 1}^\infty \frac{ g_1(n) }{ n^s } := \prod_{\chi\mod{D}} L(s, \chi)^{b_\chi}.
  \end{equation}
  We have~$(f \ast g_1^{-1})(p) = 0$ if~$p\nmid D$. Moreover, since~$\abs{b_\chi} \leq A$, we get~$\abs{g_1(n)} \leq \d_{AD}(n)$ for all~$n$. Therefore,
  $$ \sum_{n\geq 1} \frac{\abs{(f\ast g_1^{-1})(n)}}{n^{2/3}} = \prod_{p\mid D} \Big(1 + \O_{A,D}\Big(\frac1{p^{2/3}}\Big)\Big) \prod_{p\nmid D} \Big(1 + \O_{A,D}\Big(\frac1{p^{4/3}}\Big)\Big) = \O_{A, D}(1). $$
  This proves the first part of the lemma.
  
  For the second part, we define \(g_2\) by its Dirichlet series
  \[ \sum_{n = 1}^\infty \frac{ g_2(n) }{n^s} := \prod_{r \in (\Z/D\Z)^\times} \prod_{ \substack{ \text{\(p\) prime} \\ p \equiv r \bmod D } }\left( 1 + \frac{ v_f(r) }{p^s - 1} \right). \]
  The fact that \(g_2\) satisfies the required conditions can be shown using similar computations as above.
\end{proof}

Let us at this point also note the following result, which is an easy consequence of the proofs of Lemmas \ref{lemma:replacing f by g}~and~\ref{lemma: existence of g_1 and g_2}, and which will become useful later on.

\begin{lemma} \label{lemma: Dirichlet series of f}
  Let~\( f \in \mathcal F_D(A) \) and let \(\psi\) mod~\(q\) be a Dirichlet character.
  Then the Dirichlet series associated to~\( \psi f \) is given by
  \[ \sum_{n = 1}^\infty \frac{ \psi(n) f(n) }{n^s} = H_\psi(s) \prod_{ \chi \bmod D } L(s, \psi \chi)^{b_\chi} \quad \text{for} \quad \Re(s) > 1, \]
  where \( H_\psi(s) \) is some holomorphic function defined in~\( \Re(s) > \frac12 \) and where
  \[ b_\chi := \frac1D \sum_{ r \bmod D } v_f(r) \overline\chi(r), \]
  with \( v_f(n) \) as defined in~\eqref{eq: definition of v_f}.
  Moreover, for any fixed~\( \sigma_0 > \frac12 \), we have \( H_\psi(s) \ll 1 \) uniformly in~\( \Re(s) > \sigma_0 \), with the implicit constant depending at most on~\(\sigma_0\), \(A\) and~\(D\).
\end{lemma}
From Lemmas~\ref{lemma:replacing f by g} and \ref{lemma: existence of g_1 and g_2}, we deduce the following statement.

\begin{lemma} \label{lemma: F' or F'' implies F}
  To prove Proposition~\ref{prop:bound-largechars} in full generality, it suffices to prove it under either one of the additional hypotheses~\( f\in \mathcal F_D^\d(A) \) or~\( f\in \mathcal F_D^\omega(A) \).
\end{lemma}

\section{Combinatorial identites for~\texorpdfstring{\( \tau_z(n) \)}{tau\_z(n)}} \label{section: combinatorial identities}

In this section we describe the two combinatorial identites for the generalized divisor function~\( \tau_z \) on which the proofs of Theorem~\ref{th:main} are based.

\subsection{A generalization of Heath-Brown's identity} \label{section: generalization of Heath-Brown's identity}

We first derive a combinatorial decomposition analogous to~\cite{HB} for the function~$n\mapsto \d_\alpha(n)$ in the case~$\alpha\in\Q$. Our argument is based on the following polynomial identity.

\begin{lemma} \label{lemma:polynomial identity}
  Let \(u\) and \( v\) be integers such that \( v > u \geq 0 \).
  Let \( K \geq 1 \) and \( N \geq 0 \).
  Then there exist rational coefficients \(a_m\)~and~\(b_\ell\) such that there holds
  \begin{align} \label{eq:polynomial identity}
    \sum_{ K \leq m \leq (K + N) v - u } a_m ( X - 1 )^m = 1 + X^{Nv} \sum_{ 1 \leq \ell \leq K } b_\ell X^{ \ell v - u }.
  \end{align}
  The coefficients \((b_\ell)\) are unique and given explicitly by
  \begin{align} \label{eq:definition of b_l}
    b_\ell = \frac{ (-1)^\ell }{ (\ell - 1)! (K - \ell)! } \prod_{ \substack{ 1 \leq j \leq K \\ j \neq \ell } } \left( j + N - \frac uv \right).
  \end{align}
\end{lemma}
\begin{proof}
  An identity of the form~\eqref{eq:polynomial identity} exists if and only if we can find \( b_1, \ldots, b_ K \) such that the first~\(K - 1\) derivatives of the polynomial on the right hand side of~\eqref{eq:polynomial identity} vanish at~\( X = 1 \).
  This is equivalent to saying that the \( b_1, \ldots, b_K \) solve the equation
  \begin{align} \label{eq:linear equation for the b_l}
    \begin{pmatrix} 1 & \cdots & 1 \\ v + Nv - u & \cdots & Kv + Nv - u \\ \vdots & \ddots & \vdots \\ (v + Nv - u)^{K - 1} & \cdots & (Kv + Nv - u)^{K - 1} \end{pmatrix} \begin{pmatrix} b_1 \\ b_2 \\ \vdots \\ b_K \end{pmatrix} = \begin{pmatrix} -1 \\ 0 \\ \vdots \\ 0 \end{pmatrix}.
  \end{align}
  Let~$C$ be the matrix on the left, and \( B_\ell \) the same matrix but with the upper row and the \(\ell\)-th column removed. Note that~$C$ is a Vandermonde matrix, and~$B_\ell$ is a product of a Vandermond matrix with a diagonal matrix. Hence, we deduce
  \begin{align*}
    \det C &= \prod_{ 1 \leq i < j \leq K } ( j v - i v) = 2! \, 3! \, \cdots \, (K - 1)! \, v^\frac{ K (K - 1) }2, \\
    \det B_\ell &= \prod_{ \substack{ 1 \leq i < j \leq K \\ i, j \neq \ell } } ( j v - i v ) \prod_{ \substack{ 1 \leq j \leq K \\ j \neq \ell } } (jv + Nv - u).
  \end{align*}
  Since~$\det C\neq 0$, we obtain by Cramer's rule that there is a unique solution~$(b_\ell)$, given by
  \begin{align} \label{eq:b_l according to Cramer's rule}
    b_\ell = (-1)^\ell \frac{ \det B_\ell }{ \det C },
  \end{align}
  which yields \eqref{eq:definition of b_l}.
\end{proof}

\begin{theorem}\label{th:combinatorial-identity}
  Let \(v>0\) and \(r\) be integers such that \( v > u \geq 0 \) and \( r \geq 0 \).
  Let~\( K \geq 1 \) and \( x \geq 1 \).
  Then for any~\( n \leq x \), we have
  \begin{align}
    \d_{ r + \frac uv }(n) &= \sum_{\ell = 1}^K c^+_\ell \mathop{\sum \cdots \sum}_{ \substack{ m_1 \cdots m_{\ell + r} n_1 \cdots n_{\ell v - u} = n \\ n_1, \ldots, n_{\ell v - u} \leq x^{1/K} } } \d_{-\frac1v}( n_1 ) \cdots \d_{-\frac1v}( n_{\ell v - u} ), \label{eq:decomposition of d_alpha, positive case}
    \intertext{and, for~\( r \geq 1 \),}
    \d_{ -r + \frac uv }(n) &= \sum_{\ell = 1}^K c^-_\ell \mathop{\sum \cdots \sum}_{ \substack{ m_1 \cdots m_{\ell - 1} n_1 \cdots n_{\ell v + (r - 1)v - u} = n \\ n_1, \ldots, n_{\ell v + (r - 1)v - u} \leq x^{1/K} } } \d_{-\frac1v}( n_1 ) \cdots \d_{-\frac1v}( n_{\ell v + (r - 1)v - u} ), \label{eq:decomposition of d_alpha, negative case}
  \end{align}
  where the \( c^+_\ell \) and \( c^-_\ell \) are certain rational numbers, which can be bounded by
  \[ c^+_\ell, c^-_\ell \ll 1 \quad \text{for} \quad 1 \leq \ell \leq K, \]
  the implicit constant depending only on \(K\)~and~\(r\).
\end{theorem}
\begin{proof}
  Let
  \[ G(s) := \sum_{n = 1}^\infty \frac{ \d_{ -1/v }(n) g(n) }{ n^s } \quad \text{with} \quad g(n) := \begin{cases} 1 &\text{if} \quad n \leq x^{1/K}, \\ 0 &\text{otherwise.} \end{cases} \]
  We first look at~\eqref{eq:decomposition of d_alpha, positive case}.
  Here we use Lemma~\ref{lemma:polynomial identity} with~\( N = 0 \) and \( X = \zeta(s)^\frac1v G(s) \), and then multiply both sides by~\( \zeta(s)^{ r + \frac uv } \), which leads to the identity
  \[ \sum_{ K \leq m \leq K v - u } a_m ( \zeta(s)^\frac1v G(s) - 1 )^m \zeta(s)^{ r + \frac uv } = \zeta(s)^{r + \frac uv} + \sum_{ 1 \leq \ell \leq K } b_\ell \zeta(s)^{r + \ell} G(s)^{ \ell v - u }. \]
  Then \eqref{eq:decomposition of d_alpha, positive case} follows by comparing the Dirichlet coefficients on both sides.
  
  In order to show~\eqref{eq:decomposition of d_alpha, negative case}, we use Lemma~\ref{lemma:polynomial identity} with the same~\(X\) as before and with~\( N = r - 1 \), and then multiply both sides by~\( \zeta(s)^{-r + \frac uv} \).
  This gives 
  \begin{multline*}
    \sum_{ K \leq m \leq (K + r - 1) v - u } a_m \left( \zeta(s)^\frac1v G(s) - 1 \right)^m \zeta(s)^{-r + \frac uv} = \zeta(s)^{-r + \frac uv} + \sum_{ 1 \leq \ell \leq K } b_\ell \zeta(s)^{\ell - 1} G(s)^{ \ell v + (r - 1) v - u },
  \end{multline*}
  and \eqref{eq:decomposition of d_alpha, negative case} follows again by comparing the Dirichlet coefficients on both sides.
\end{proof}

\begin{remark}
  With \( r = v = 1 \) and \(u = 0\), the identity~\eqref{eq:decomposition of d_alpha, negative case} leads to the decomposition of~\( \mu(n) \) described in~\cite[(13.38)]{IK04}.
\end{remark}

\subsection{A combinatorial identity of Linnik's type} \label{section: generalization of Linnik's identity}

Here we derive a combinatorial decomposition for~$\d_z$ using an approach analogous to~\cite{Linnik}.

We denote by~$P^+(n)$ the largest, and by~$P^-(n)$ the smallest prime factor of an integer~$n>1$, with the convention that~$P^+(1)=1$ and~$P^-(1)=\infty$.
Given an arbitrary multiplicative function~$f$ and a complex number~$z\in\C$, we define the $z$-fold convolution of~\(f\) as the multiplicative function given by
$$ f^{ (\ast z) }(p^\nu) := \sum_{1\leq r\leq \nu} \binom zr \ssum{\lambda_1, \dotsc, \lambda_r \geq 1 \\ \lambda_1 + \dotsb + \lambda_r = \nu} f(p^{\lambda_1}) \dotsb f(p^{\lambda_r}) \qquad (\nu\geq 1). $$
Note that~$f^{ (\ast z) }(p) = z f(p)$, and that for~$\ell\in\N$ the \(\ell\)-fold convolution as defined here coincides with the \(\ell\)-fold convolution defined in the traditional sense.
We will be eventually interested in the case when~$f=\chi$ is a Dirichlet character, in which case we have~$f^{ (\ast z) } = \d_z^{\chi}$.

\begin{theorem}\label{th:combinatorial-identity-linnik}
  Let~\( K\in\N_{>0}\) and~\( A, x \geq 1 \).
  Then for all~\(z\in\C\) there exist complex numbers~$(c_{\ell})_{0\leq \ell \leq K}$, such that for all~\( x \geq 1 \) and all multiplicative functions~$f$, we have the following identity for~\( n \leq x \),
  \begin{align}
    f^{ (\ast z) }(n) = \sum_{0\leq \ell < K} c_\ell \ssum{n = n_1 n_2 \\ P^+(n_1) \leq x^{1/K}}  f^{ (\ast (z-\ell)) }(n_1) f^{ (\ast \ell) }(n_2),
  \end{align}
  where the coefficients~\( c_\ell \) can be bound by~$c_\ell = O_{K, A}(1)$ uniformly for~$|z|\leq A$.
\end{theorem}
\begin{proof}
  Let~\( y := x^{1/K} \).
  We may certainly assume that~$f(p^k)$ vanishes if~$p>x$. For~$\Re(s)$ large enough the function~$\log F(s)$ is well defined, and we have~$F(s)^z = \sum_n f^{ (\ast z) }(n) n^{-s}$. Let
  $$ F(s,y) = \prod_{p\leq y}\Big(\sum_{k\geq 0} \frac{f(p^k)}{p^{ks}}\Big), \qquad G(s,y) = \prod_{p > y}\Big(\sum_{k\geq 0} \frac{f(p^k)}{p^{ks}}\Big). $$
  For~$\Re(s)>0$, the decomposition~$F(s) = F(s, y) G(s, y)$ yields
  \begin{align*}
    F(s)^z ={}& F(s, y)^z (1 + (G(s, y) - 1))^z \\
    ={}& F(s, y)^z \sum_{k\geq 0} \binom zk (G(s, y)-1)^k \\
    ={}& F(s, y)^z \sum_{0 \leq k < K} \binom zk (G(s, y)-1)^k + R(s)
  \end{align*}
  with
  \[ R(s) := F(s, y)^z \sum_{k \geq K} \binom zk (G(s, y)-1)^k. \]
  Note that the series converge absolutely if~$\Re(s)$ is large enough in terms of~$f$. By expanding, we get
  $$ F(s)^z = F(s, y)^z \sum_{0 \leq \ell < K} c_\ell G(s, y)^\ell + R(s), $$
  with
  $$ c_\ell := (-1)^\ell \sum_{\ell\leq k < K} (-1)^k \binom zk \binom k\ell. $$
  We read the coefficients of~$n^{-s}$, for~$n\leq x$, on each side. Note that for~$k\geq K$, the series~$(G(s, y)-1)^k$ has no corresponding Dirichlet coefficients, so there there is no contribution from~$R(s)$. The claimed equality follows on writing~$G(s, x^{1/K}) = F(s) F(s, x^{1/K})^{-1}$.
\end{proof}

\begin{remarks}\
  \begin{itemize}[leftmargin=2em]
    \item Compared with~\eqref{eq:decomposition of d_alpha, positive case}--\eqref{eq:decomposition of d_alpha, negative case}, this identity has the significant advantage that it is uniform for~$z\ll 1$ complex.
    \item The case~$K=2$ only involves the exponents~$\ell\in\{0, 1\}$. It follows, for instance, that if~$f^{ (\ast z) }$ satisfies a Siegel-Walfisz estimate (in the sense of~\cite[eq. (1.2)]{GranvilleShao18}), and if~$f$ satisfies a Bombieri-Vinogradov theorem, then~$f^{ (\ast z) }$ satisfies a Bombieri-Vinogradov theorem as well.
    \item The case~$K=2$,~$f=\1$ leads to Eratosthenes' sieve identity: for all~$n\in(\sqrt{x}, x]$, we have
      $$ \1_{n\text{ prime}} = \ssum{d\mid n \\ p\mid d \Rightarrow p\leq \sqrt{x}} \mu(d). $$
      For any~$\eta\in(0,1/2)$, either we have~$d\leq x^\eta$ (which corresponds to type~I sums), or~$d>x^\eta$, in which case we can localize a factor of~$d$ in the interval~$[x^\eta, x^{1/2+\eta}]$ (and this corresponds to type~II sums).
  \end{itemize}
\end{remarks}

The main property which allows Theorem~\ref{th:combinatorial-identity-linnik} to be used in our arguments is the following factorization lemma, in the spirit of Lemma~3.1 of~\cite[p.29]{Vaughan1989}; see~\cite{Hmyrova1964} for an early use of this property, and~\cite{FouvryTenenbaum1996} for an application in a context similar to ours.

\begin{lemma}\label{lem:friable-well-fact}
  For any multiplicative function~$f:\N\to\R$, any compactly supported function~$g:\N\to\C$, and all~$y, w \geq 2$, we have
  \begin{equation} \label{eq:wellfactorization}
    \sum_{P^+(n)\leq y} f(n) g(n) = \Sigma_{\text{triv}} + \Sigma_\text I + \mathcal{O}\left( \Sigma_\text{II} \right),
  \end{equation}
  where
  \begin{align*}
    \Sigma_\text I &= \ssum{n\leq w \\ P^+(n) \leq y} f(n) g(n), \qquad\qquad \Sigma_{\text{triv}} = \ssum{n > w \\ P^+(n) \leq y \\ \exists p^\nu \| n, p^\nu > y} f(n) g(n), \\
    \Sigma_\text{II} &= (\log y) \sup_{\alpha, \beta} \Big| \sum_{w < m \leq yw} \sum_n \alpha_m \beta_n g(mn) \Big|,
  \end{align*}
  the supremum in~\( \Sigma_\text{II} \) being taken over all sequences~$(\alpha_m), (\beta_n)$ of complex numbers satisfying
  $$ |\alpha_m| \leq |f(m)|, \quad |\beta_n| \leq (\abs{f}\ast\abs{f})(n). $$
\end{lemma}
\begin{proof}
  In an integer~$n$ with~$P^+(n)\leq y$ is not counted in the first two sums on the right-hand-side, then~$n>w$ and all prime powers~$p^\nu\| n$ satisfy~$p^\nu \leq y$. By incorporating these prime powers as~$p$ increases, we may factor~$n = n_1 n_2$ uniquely in such a way that
  $$ P^+(n_1) < P^-(n_2), \quad w < n_1 \leq w Q^+(n_1), $$
  where~$Q^+(n_1)$ is the prime power corresponding to the largest prime of~$n_1$: $Q^+(n_1) = P^+(n_1)^\nu \| n_1$. Note that this implies~$(n_1, n_2)=1$. Our statement follows after separating variables~\cite[Lemma~13.11]{IK04} in the condition~$P^+(n_1) < P^-(n_2)$.
\end{proof}

\section{Auxiliary estimates} \label{section: auxiliary estimates}

In this section we collect some estimates on~\( \Delta_h(n, R) \), which will be needed in the following sections.

\subsection{The second moment of~\texorpdfstring{\( \Delta_h(n; R) \)}{Delta\_h(n; R)}}

On several occasions, we will require the following rough upper-bound for the ``main terms''.
\begin{lemma} \label{lem:rough-bound-sum-d-dt}
  For~$x\geq 3$, \( R \geq 1 \) and $(a, h)\in\Z^2$ such that~$1\leq a, |h|, R \leq x^{1/4}$, the following estimate holds,
  $$ \sum_{ \frac x2 < n \leq x } | \Delta_h(an; R) |^2 \ll \d((a, h))^2 x(\log x)^4. $$
\end{lemma}
\begin{proof}
  We have
  \[ \sum_{ \frac x2 < n \leq x } \left| \Delta_h(an; R) \right|^2 \ll \sum_{ \frac x2 < n \leq x } \d(an-h)^2 + \sum_{ \frac x2 < n \leq x } |\dt_h(an;R)|^2 =: G_1 + G_2, \]
  and we now proceed to estimate the two sums \( G_1 \)~and~\( G_2 \) separately.
  
  We first look at~\(G_1\).
  For notational convenience, let
  \[ a' := \frac a{ (a, h) }, \quad h' := \frac h{ (a, h) } \quad \text{and} \quad t := (a, h). \]
  We start by splitting the sum according to the size of~\( t^\ast = ( an - h, t^\infty ) \) as follows,
  \begin{align*}
    G_1 &= \sum_{ \substack{ t^\ast \mid t^\infty \\ (t^\ast, a') = 1 \\ t^\ast \leq x^{1/2} } } \sum_{ \substack{ \frac x2 < n \leq x \\ a' n \equiv h' \bmod t^\ast \\ ( \frac{ a' n - h' }{ t^\ast }, t ) = 1 } } \d( a n - h )^2 + \sum_{ \substack{ t^\ast \mid t^\infty \\ (t^\ast, a') = 1 \\ t^\ast > x^{1/2} } } \sum_{ \substack{ \frac x2 < n \leq x \\ a' n \equiv h' \bmod t^\ast \\ ( \frac{ a' n - h' }{ t^\ast }, t ) = 1 } } \d( a n - h )^2
      =: G_{1 \text a} + G_{1 \text b}.
  \end{align*}
  In order to estimate~\( G_{1 \text a} \) we choose \( b, y \in \Z \) such that~\( a' b = 1 + y t^\ast \) and write
  \begin{align*}
    G_{1 \text a} &= \sum_{ \substack{ t^\ast \mid t^\infty \\ (t^\ast, a') = 1 \\ t^\ast \leq x^{1/2} } } \d(t^\ast t)^2 \sum_{ \substack{ \frac{ x - 2 b h' }{ 2 t^\ast } < n' \leq \frac{ x - b h' }{ t^\ast } \\ ( y h' + n' a', t ) = 1 } } \d( y h' + n' a' )^2 \\
      &\leq \sum_{ \substack{ t^\ast \mid t^\infty \\ (t^\ast, a') = 1 \\ t^\ast \leq x^{1/2} } } \d(t^\ast t)^2 \sum_{ \substack{ \frac{ a' x - 2 h' }{ 2 t^\ast } < m \leq \frac{ a' x - h' }{ t^\ast } \\ m \equiv y h' \bmod a' } } \d( m )^2.
  \end{align*}
  The sum over~\(m\) can now be estimated via~\cite[Theorem~2]{Shiu} or~\cite[Theorem~1]{BaVe}, which leads to
  \begin{equation} \label{eq: bound for G_1a}
    G_{1 \text a} \ll x \log^3 x \sum_{ \substack{ t^\ast \mid t^\infty \\ t^\ast \leq x^{1/2} } } \frac{ \d(t^\ast t)^2 }{t^\ast} \ll \d( (a, h) )^2 x \log^4 x.
  \end{equation}
  In~\( G_{1 \text b} \) we bound all the summands trivially and get
  \[ G_{1 \text b} \ll \sum_{ \substack{ t^\ast \mid t^\infty \\ (t^\ast, a') = 1 \\ t^\ast > x^{1/2} } } \sum_{ \substack{ \frac x2 < n \leq x \\ a' n \equiv h' \bmod t^\ast } } \d( t ( a' n - h' ) )^2 \ll x^{1 + \varepsilon} \sum_{ \substack{ t^\ast \mid t^\infty \\ x^{1/2} < t^\ast \leq 2 a' x } } \frac1{ t^\ast } \ll x^{ \frac34 + \varepsilon }, \]
  so that together with~\eqref{eq: bound for G_1a} we deduce
  \[ G_1 \ll \d( (a, h) )^2 x \log^4 x. \]
  
  Next we look at~\( G_2 \).
  Here we first rewrite~\( \dt_h(an; R) \) as
  \[ \dt_h(an; R) := 2 \ssum{ \alpha \mid (a, h) } \ssum{ \delta \mid ( \frac h{ \alpha }, n ) \\ ( \delta, \frac a\alpha ) = 1 } \ssum{ q \leq \frac{ \sqrt{an - h} }{ \alpha \delta } } \frac1{ \vphi(q) } \ssum{ \chi \mod q \\ \cond(\chi) \leq R } \chi( \tfrac h{\alpha \delta} ) \overline{ \chi( \tfrac{an}{ \alpha \delta } ) }, \]
  so that after expanding the square we are led to 
  \begin{equation*}
    G_2 \leq 4 \ssum{ \alpha_1, \alpha_2 \mid (a, h) \\ \delta_1 \mid \frac h{ \alpha_1 }, \,\, \delta_2 \mid \frac h{ \alpha_2 } } \ssum{ q_1 \leq \frac{ \sqrt{ax - h} }{ \alpha_1 \delta_1 } \\ q_2 \leq \frac{ \sqrt{ax - h} }{ \alpha_2 \delta_2 } } \frac1{ \vphi(q_1) \vphi(q_2) } \ssum{ \chi_1 \mod{q_1} \\ \chi_2 \mod{q_2} \\ \cond(\chi_1), \cond(\chi_2) \leq R } S\left( \overline{\chi_1} \chi_2, \tfrac x{ [\delta_1, \delta_2] } \right),
  \end{equation*}
  with
  \[ S(\chi, y) := \max_{ \frac y2 \leq y_0 < y } \bigg| \ssum{ y_0 < n \leq y } \chi(n) \bigg|. \]
  If~$\chi_1$ and~$\chi_2$ are induced by the same primitive character, we use the trivial bound~$S\left( \overline{\chi_1} \chi_2, y \right) \leq y$. Otherwise, the P\'olya-Vinogradov bound applies and~$S\left( \overline{\chi_1} \chi_2, y \right) \ll \d(q_1q_2) R \log R$. Inserting these bounds, we eventually obtain
  $$ G_2 \ll \d( (a, h) )^2 x \log^4 x + x^\varepsilon R^3 \ll \d( (a, h) )^2 x \log^4 x $$
  by our assumption~$R\leq x^{\frac14}$. This concludes the proof.
\end{proof}

\subsection{Comparison of main terms}

We begin by two technical lemmas related to the main terms that will appear later.
Let~\( X \geq 1 \), and let~$f$ and~$v$ be two smooth functions which are both compactly supported inside~$\R_+^*$.
We assume that \( \operatorname{supp} f \subset [C_1 X, C_2 X] \), where \(C_1\)~and~\(C_2\) are some fixed constants, and that for some~$\Omega\in (0, 1]$, we have
$$ \|v^{(j)}\|_\infty \ll_j 1, \qquad  \|f^{(j)}\|_\infty \ll_j (\Omega X)^{-j} \qquad \int_\R |f^{(j+1)}| \ll (\Omega X)^{-j}, $$
for all~$j\geq 0$.
Furthermore, we define
\begin{equation} \label{eq:def-main-term-smooth}
  M_{f, v}(b, h; M) := \frac1b \sum_{d \mid b} \frac{ c_d(h) }d \int \! ( \log(\xi - h) + 2\gamma - 2 \log d ) f(\xi) v\left( \frac\xi{bM} \right) \, d\xi,
\end{equation}
where
\[ c_d(h) := \sum_{\nu\mods{d}}\e(\nu h/d) = \sum_{\delta|(h, d)} \delta \mu(d/\delta) \]
denotes the Ramanujan sum.

\begin{lemma}\label{lem:compare-MTsmooth-dt}
  For~$(b, h)\in\Z^2$, $b, M\geq 1$, and~$R\geq 1$, we have
  $$ \sum_m f(bm) v\Big(\frac mM\Big) \dt_h(bm; R) = M_{f, v}(b, h; M) + \O\Big(X^\ee R^{3/2} + X^{1/2+\ee} \frac{(h, b)}{b}\Big). $$
\end{lemma}

\begin{proof}
Recall the definition~\eqref{eq:def-dthR}. By partial summation and the P\'olya-Vinogradov inequality, we have
$$ \sum_m f(bm) v\Big(\frac mM\Big) \dt_h(bm; R)  = 2\sum_m f(bm)v\Big(\frac mM\Big) \ssum{q\leq \sqrt{bm-h} \\ (bm, q) = (h, q)} \frac1{\vphi(\frac q{(q,h)})} + \O(X^\ee R^{3/2}). $$
The condition~$(bm, q) = (h, q)$ in the sum on the right-hand side is equivalent to
$$ (b, q)|h, \quad \frac{(h, q)}{(b, q)}\Big\vert m, \quad \Big(\frac{m(b, q)}{(h, q)}, \frac{q}{(h, q)}\Big) = 1. $$
Using Möbius inversion and our hypotheses on~$f$ and~$v$, we can replace the~$m$-sum by the corresponding integral and obtain
$$ \sum_m f(bm) v\Big(\frac mM\Big) \dt_h(bm; R) = \frac2b \int f(\xi)v\Big(\frac\xi{bM}\Big) \ssum{q\leq\sqrt{\xi-h} \\ (b, q)|h} \frac{(b, q)}q\dd \xi + \O(X^\ee R^{3/2}). $$
The main term on the right-hand side may be rewritten as
$$ \frac2b \sum_{d|b} \frac{c_d(h)}{d} \int f(\xi)v\Big(\frac{\xi}{bM}\Big) H\Big(\frac{\sqrt{\xi-h}}d\Big) \dd \xi + \O(X^\ee R^{3/2}) $$
where~$H(x) = \sum_{q\leq x}1/q = \log x + \gamma + \O(x^{-1})$. This gives the claimed estimate.
\end{proof}

Next, we define
\begin{equation}
  M_{f, v}^\chi(b, h; M) := \ssum{a\bmod{D} \\ (a, D)=1} \chi(a) M_{f_{ab}, v_{a/M}}(bD, h-ab; M/D)\label{eq:def-MTsmooth-chi}
\end{equation}
where~$f_{ab}(\xi) := f(\xi + ab)$ and~$v_{a/M}(\xi) := v(\xi+a/M)$.
\begin{lemma}\label{lem:compare-MTsmooth-chi-dt}
  If~$b=b^\circ b^*$ with~$b^\circ | D^\infty$ and~$(b^*, D)=1$, then
  \begin{equation}
    M_{f,v}^\chi(b,h;M) = \frac1{bD}\chi\Big(\frac{h}{(h, b)}\Big) {\bar\chi}\Big(\frac{b}{(h, b)}\Big) \sum_{d|b^*} \frac{c_d(h)}{d}
    \int \Big(\log\Big(\frac{\xi-h}{(Db^\circ d)^2}\Big) + 2\gamma\Big) f(\xi)v\Big(\frac{\xi}{bM}\Big)\dd\xi.\label{eq:rewrite-Mfvchi}
  \end{equation}
  Moreover, if~$\chi\bmod{D}$ is primitive, we have
  \begin{align*}
    \sum_m f(bm) &{} v\Big(\frac mM\Big) \chi(m) \dt_h(bm; R) \\
    {}& = M_{f, v}^\chi(b, h; M) + \O\Big(\1_{D>R}(b, h)\frac{X(\log X)^3}{bD} + X^\ee D^{1/2}R^{3/2} + X^{1/2+\ee} \frac{(h, b^*)}{b^*}\Big)
  \end{align*}
  where~$\1_{D> R}=1$ if~$D>R$ and~$0$ otherwise.
\end{lemma}

\begin{proof}
We rewrite
$$ M_{f, v}^\chi(b, h; M) = \frac1{bD}\ssum{a\mod{D} \\ (a, D)=1} \chi(a) \sum_{d|bD} \frac{c_d(h-ab)}{d} \int (\log(\xi-h) + 2\gamma - 2\log d) f(\xi)v\Big(\frac{\xi}{bM}\Big)\dd\xi. $$
Using Gau{\ss} sums,
$$ \ssum{a\mod{D} \\ (a, D)=1}\chi(a) c_d(h-ab) = \gsm(\chi) {\bar \chi}\Big(\frac{-bD}{d}\Big) \sum_{\nu\mods{d}} {\bar \chi}(\nu) \e\Big(\frac{h\nu}d\Big). $$
This last expression vanishes unless~$D(b, D^\infty)|d$. Denoting~$b^\circ = (b, D^\infty)$ and~$b^* = b/b^\circ$, we obtain for~$d|b^*$
\begin{align*}
  \ssum{a\mod{D} \\ (a, D)=1}\chi(a) c_{Db^\circ d}(h-ab) {}& = \gsm(\chi) {\bar\chi}(-b^*/d) \gsm(\bar\chi) \sum_{\delta|(b^\circ d, h)} \delta \chi(h/\delta) \mu(b^\circ d/\delta) {\bar\chi}(b^\circ d/\delta) \\
  {}& = b^\circ D \1_{b^\circ | h} {\bar \chi}(b^*) \chi(h/b^\circ) c_d(h) \\
  {}& = b^\circ D \chi\Big(\frac{h}{(h, b)}\Big) {\bar\chi}\Big(\frac{b}{(h, b)}\Big) c_d(h).
\end{align*}
This yields our first claim.

For the second, the computations are similar to the previous Lemma. If~$D>R$, we get
\begin{equation}
\sum_m f(bm) v\Big(\frac mM\Big) \chi(m) \dt_h(bm; R) \ll X^\ee D^{1/2} R^{3/2},\label{eq:bound-typed1-charsnonprinc}
\end{equation}
while on the other hand~$M_{f,v}^\chi(b, h ; M) \ll (b, h) (bD)^{-1} X (\log X)^2$ by a simple computation from~\eqref{eq:rewrite-Mfvchi}. If~$D\leq R$, the bound~\eqref{eq:bound-typed1-charsnonprinc} applies to all the characters involved in the definition of~$\dt_h(bm;R)$, except all those which are induced by~$\chi$. We obtain
$$ \sum_m f(bm) v\Big(\frac mM\Big) \chi(m) \dt_h(bm; R) = 2\ssum{D|\frac q{(q, h)}} \frac{\chi(\frac h{(b, q)}){\bar\chi}(\frac b{(b,q)})}{\vphi(\frac q{(h,q)})} \ssum{(bm,q)=(h,q) \\ q^2 \leq bm-h} f(bm) v\Big(\frac mM\Big) + \O(X^\ee D^{1/2}R^{3/2}). $$
Similarly as above, the main term in the right-hand side can be rewritten
$$ \frac2b\chi\Big(\frac{h}{(h,b)}\Big){\bar\chi}\Big(\frac{b}{(h,b)}\Big) \int f(\xi)v\Big(\frac{\xi}{bM}\Big) \ssum{q\leq \sqrt{\xi-h} \\ (b,q)|h, \,\, D| \frac q{ (h,q) } \\ ( D, (b,h) / (b,q) )=1} \frac{(q,b)}{q} \dd \xi + \O(X^\ee) $$
The~$\chi$-factors impose the conditions~$b^\circ|h$ and~$(D, h/b^\circ)=1$. We rewrite the~$q$-sum as
$$ \ssum{q\leq \sqrt{\xi-h} \\ (b,q)|h, \,\, D| \frac q{ (h,q) } \\ (D, (b,h)/(b,q))=1} \frac{(b, q)}{q} = \frac1D \ssum{q\leq \sqrt{\xi-h}/(Db^\circ) \\ (b^*, q) | h} \frac{(q, b^*)}{q} = \frac1D \sum_{d|b^*}\frac{c_d(h)}{d}H\Big(\frac{\sqrt{\xi-h}}{Db^\circ d}\Big) $$
whence the claimed expression.
\end{proof}

\subsection{Type~\texorpdfstring{$\d_1$}{tau\_1} estimates}

The following estimate is relevant for convolutions with one smooth component of size~$\gg x^{1/3+\ee}$. It can be viewed as a generalization of a result of Selberg~\cite[p.235]{Selberg-Collected} on the equidistribution of~$\d_2$ in arithmetic progressions.

\begin{lemma} \label{lem:Delta for large M}
  Let~\( \varepsilon > 0 \), let~\( C_2 > C_1 > 0 \), let \( v : (0, \infty) \to \R \) be a smooth and compactly supported function, and let \(\chi\bmod{D}\) be a Dirichlet character of modulus~$D\geq 1$.
  Then we have, for any \( X, M \geq 1 \) and~$R\geq D$, any \( 1 \leq bD , |h| \ll X^{1 - \varepsilon} \), and any interval~\( I \subset [ C_1 X,  C_2 X ] \),
  \begin{align} \label{eq:Delta for large M, smooth version}
    \sum_{ m : \,\, bm \in I } \chi(m) v\left( \frac mM \right) \Delta_h(bm; R) & \ll X^\varepsilon \left( D X^\frac13 + (b, hD^\infty)M X^{-\frac12} + D^{\frac12} R^\frac32 \right).
  \end{align}
  The implied constants depend only on the function~\(v\) and the constants~\( \varepsilon \), \( C_1 \)~and~\(C_2\).
\end{lemma}
\begin{proof}
  Note that we can always assume~\( bM \asymp X \), since otherwise the sums in consideration are empty.
  Let \( f : (0, \infty) \to [0, \infty) \) be a smooth weight function, which is compactly supported in \( \operatorname{supp} f \subset [ C_1 X / 2, 2 C_2 X ] \), which has value~\( f(\xi) = 1 \) for all~\( \xi \in I \), and whose derivatives satisfy
  \[ f^{ (\nu) }(\xi) \ll \frac1{ (\Omega X)^\nu } \quad \text{for} \quad \nu \geq 0 \quad \text{and} \quad \int \! \left| f^{ (\nu) }(\xi) \right| \, d\xi \ll \frac1{ (\Omega X)^{\nu - 1} } \quad \text{for} \quad \nu \geq 1, \]
  for some constant~\( \Omega \leq 1 \).
  We can then encode the condition~\( bm \in I \) by using the function~\( f(\xi) \) via
  \begin{equation} \label{eq:smoothing the Delta sum}
    \sum_{ m : \,\, bm \in I } \chi(m) v\left( \frac mM \right) \Delta_h(bm; R) = \sum_m f(bm) v\left( \frac mM \right) \chi(m) \Delta_h(bm; R) + \O\left( \Omega X^{1 + \varepsilon} b^{-1} \right),
  \end{equation}
  so that it suffices to consider the smoothed sum on the right hand side.
  
  Assume first that \(\chi\) is the trivial character.
  In~\cite[Section~3]{T-dk} it is shown that
  \begin{equation} \label{eq:evalution of d_2 in arithmetic progressions}
    \sum_m f(bm) v\left( \frac mM \right) \d(bm - h) = M_{f, v}(b, h; M) + \O\left( X^\varepsilon b^\frac12 \Omega^{-\frac12} \right),
  \end{equation}
  where the main term~\( M_{f, v}(b, h; M) \) is given by~\eqref{eq:def-main-term-smooth}. By Lemma~\ref{lem:compare-MTsmooth-dt}, we obtain
  \begin{align}
    \sum_m f(bm) v\left( \frac mM \right) \dt_h(bm; R) = M_{f, v}(b, h; M) + \O\left( X^\varepsilon R^\frac32 + (b, h) b^{-1} X^{\frac12 + \varepsilon} \right). \label{eqn: evaluation of tilde d}
  \end{align}
  The estimate~\eqref{eq:Delta for large M, smooth version}, in the case~$D=1$ and~$\chi=\1$, now follows from~\eqref{eq:smoothing the Delta sum} with the choice~\( \Omega = b X^{-\frac23} \).
  
  Now assume that \(\chi\) is a primitive character modulo~$D$, where~$D\leq R$ and~$bD\ll X^{1-\ee}$.
  We write
  \[ \sum_m f(bm) v\left( \frac mM \right) \chi(m) \d(bm - h) = \ssum{a\mod{D}\\(a, D)=1} \chi(a) \left( \sum_m \tilde f(\tilde b m) \tilde v\left( \frac m{\tilde M} \right) \d(\tilde b m - \tilde h) \right), \]
  with
  \[ \tilde b := Db, \quad \tilde M := M/D, \quad \tilde h := h - ab, \quad \tilde f(\xi) := f(\xi + ab) \quad \text{and} \quad \tilde v(\xi) := v\left( \xi + \frac aM \right), \]
  so that we can use our former result~\eqref{eq:evalution of d_2 in arithmetic progressions} to get
  \[ \sum_m f(bm) v\left( \frac mM \right) \chi(m) \d(bm - h) = M_{f, v}^\chi(b, h; M) + \O\!\left( X^\varepsilon D^{\frac 32}b^\frac12 \Omega^{-\frac12} \right), \]
  where \( M_{f, v}^\chi(b, h; M) \) is defined in~\eqref{eq:def-MTsmooth-chi}. By Lemma~\ref{lem:compare-MTsmooth-chi-dt}, we obtain
  \begin{align*}
    \sum_m f(bm) v\left( \frac mM \right) \chi(m) \d(bm - h) {}& = \sum_m f(bm) v\left( \frac mM \right) \chi(m) \dt_h(bm; R) \\
    {}& \qquad \qquad + O_\ee\Big(X^\ee D^{3/2}b^{1/2}\Omega^{-1/2} + X^\ee D^{1/2}R^{3/2} + X^{1/2+\ee}\frac{(h, b^*)}{b^*}\Big)
  \end{align*}
  We choose \( \Omega = b D X^{-\frac23} \), and hence get~\eqref{eq:Delta for large M, smooth version} also in this case.
  
  The case when~$\chi$ is not necessarily primitive follows at once using Möbius inversion.
\end{proof}

\subsection{Type~\texorpdfstring{$\d_2$}{tau\_2} estimates}

The following estimate is a uniform version of the~$\d_2-\d_2$ shifted convolution problem obtained recently by the second author.

\begin{lemma} \label{lem:Delta for large M_1, M_2}
  Let~\( \varepsilon > 0 \), let~\( C_2 > C_1 > 0 \), let \( v_1, v_2 : (0, \infty) \to \R \) be smooth and compactly supported weight functions, and let \(\chi_1\)~and~\(\chi_2\) be Dirichlet characters mod~\(D\).
  Then for any~\( X, b \geq 1 \) and~\( R \geq D \), any~\( M_1 \geq M_2 \geq 1 \) with~\( X^\frac12 \leq M_1 M_2\), any~\( h \in \Z \) with~\( 1 \leq |h|, D \leq X^{1/4} \) and any interval~\( I \subset [ C_1 X/2, C_2 X ] \), we have
  \begin{multline} \label{eq:Delta for large M_1, M_2, smooth version}
    \sum_{ \substack{ m_1, m_2 : \\ b m_1 m_2 \in I } } v_1\left( \frac{m_1}{M_1} \right) v_2\left( \frac{m_2}{M_2} \right) \chi_1(m_1) \chi_2(m_2) \Delta_h(b m_1 m_2; R) \\
      \ll b^\circ D^{5/2} (XM_1M_2)^{1/3+\ee}\bigg( 1 + \bigg(\frac{|h|M_1M_2}{XD}\bigg)^{1/4}\bigg) + X^{-1/2+\ee} R^{3/2} b^\circ (h,b) M_1M_2^2.
  \end{multline}
  The implied constant depends only on the constants~\( \varepsilon\), \(C_1\)~and~\(C_2\), and the functions \( v_1 \)~and~\( v_2 \).
\end{lemma}

\begin{proof}
  Note that we can make the assumption~\( b \asymp \frac X{M_1 M_2} \), as otherwise the sum in consideration is empty.
  Also, as in Lemma~\ref{lem:Delta for large M}, we can exchange the original sum by its smoothed version,
  \[ \sum_{ m_1, m_2 } f(b m_1 m_2) v_1\left( \frac{m_1}{M_1} \right) v_2\left( \frac{m_2}{M_2} \right) \chi_1(m_1) \chi_2(m_2) \Delta_h(b m_1 m_2; R), \]
  with an error of the size of~\( \O\left( \Omega X^{1 + \ee} b^{-1}  \right) \).
  
  Let~$\chi_0 := {\bar\chi_1}\chi_2$. The results of~\cite{T-dk} cannot be quoted as a black box, however, the computations of~\cite{Top17a} on which they are based may be adapted with little change.
  We write
  \[ \sum_{m_1, m_2} f(b m_1 m_2) v_1\left( \frac{m_1}{M_1} \right) v_2\left( \frac{m_2}{M_2} \right) \chi_1(m_1) \chi_2(m_2) \d(b m_1 m_2 - h) = \sum_{ a \mod{D} } \chi_1(a) D(a), \]
  where \( D(a) \) is the defined as
  \[ D(a) := \sum_n w_1\left( \frac{r_1 n + f_1}{ x_1 } \right) w_2\left( \frac{r_2 n + f_2}{ x_2 } \right) \d( r_1 n + f_1 ) \sum_{ \substack{ n_1, n_2 \\ n_1 n_2 = r_2 n + f_2 } } \chi_0(n_2) h_{M_2 M_1}(n_1, n_2), \]
  with
  \[ r_1 := bD, \quad r_2 := D, \quad f_1 := a b - h, \quad f_2 := a, \quad x_1 := X, \quad x_2 := \frac Xb, \]
  and
  \[ w_1(\xi) := \sqrt{ f(X \xi + h) }, \quad w_2(\xi) := \sqrt{ f(X \xi) }, \quad h_{M_2 M_1}(n_1, n_2) := v_2\left( \frac{n_1}{M_2} \right) v_1\left( \frac{n_2}{M_1} \right). \]
  The sum~\( D(a) \) is now of the same shape as the sum~\( D_{AB}(x_1, x_2) \) defined in~\cite[p.~157]{Top17a}, with the function~${\tilde f}(a, b)$ there replaced by~$\chi_0(a){\tilde f}(a, b)$. The computations of Section~3 of~\cite{Top17a} can then be adapted with the following changes. In Section~3.1 of~\cite{Top17a}, the expressions~$\Sigma_{AB}^0$ and~$\Sigma_{AB}^{\pm}$ have an additional factor~$\chi_0(au_2/u_2^*)$ in the summands. In the sums in the definition of~$R_{AB}^{\pm}$, p.159~\textit{ibid}, the summand has to be multiplied by an additionnal factor~$\chi_0(c)$, and the altered relation
  $$ \Sigma_{AB}^{\pm} = \ssum{u_2^*|u_2 \\ r_2^* |r_2} \chi_0\Big(\frac{u_2}{u_2^*}\Big)\ssum{d \\ (d, r_1^*s_2u_2^*)=1} \chi_0(d) \frac{R_{AB}^\pm}{d} $$
  holds. Consequently, the relationship between~$R_{AB}^\pm(N;\chi)$ and~$K_{AB}^\pm(N; \chi)$ becomes
  $$ R_{AB}^\pm(N;\chi) = \sum_{N<n\leq2N} \d(n) {\hat S}_v({\bar\chi};n)K_{AB}^\pm(\chi\chi_0;n). $$
  The rest of the argument of~\cite{Top17a} is adapted with the only change that the Kuznetsov formula is applied with nebentypus~$\chi \chi_0$ instead of~$\chi$. This has no effect on the error terms, since the bounds in Theorem~2.6 and Lemmas~2.7, 2.8 and 2.9 of~\cite{Top17a} are uniform with respect to the nebentypus.

  By the bound~(3.4) of~\cite{Top17a}, with~$b^\circ = (b, D^\infty)$, $r_0\gets Db^\circ$ and~$h\gets hD$, we obtain
  \begin{multline*}
    D(a) = \ssum{m_2 \\ (m_2,D)=1} \frac{ \chi_0(m_2) v_2\left( \frac{m_2}{M_2} \right)}{bD m_2} \sum_{d \mid bDm_2} \frac{ c_d\left(a b m_2{\bar m_2}-h \right) }d \int \! \left( \log(\xi - h) + 2\gamma - 2\log d \right) \\
    \cdot f(\xi) v_1\left( \frac\xi{b M_1 m_2} \right) \, d\xi + \O\left( b^\circ D^{3/2} X^{\frac12 + \varepsilon} \left( \frac1{ \Omega^\frac12 } + \Big(\frac{(b, h) X}{Db^2}\Big)^\theta\Big\{1 + \Big(\frac{|h|}{bD}\Big)^{1/4}\Big\} \right) \right),
  \end{multline*}
  where \( \overline m_2 \) denotes any integer such that~\( {\bar m_2} \cdot m_2 \equiv 1 \bmod D \).
  We sum over~$a\mod{D}$, exchange the~$a$- and~$m_2$-sums, and change variables~$a\gets a m_2$. We obtain
  \begin{multline*}
    \sum_{a \mod{D}} \chi_1(a) D(a) = \sum_{m_2} v\left( \frac{m_2}{M_2} \right) \chi_2(m_2) M_{f, v_1}^{\chi_1}(m_2 b, h; M_1) \\
      + \O\left( b^\circ D^{5/2} X^{\frac12 + \varepsilon} \left( \frac1{ \Omega^\frac12 } + \Big(\frac{(b, h) X}{Db^2}\Big)^\theta\Big\{1 + \Big(\frac{|h|}{bD}\Big)^{1/4}\Big\} \right) \right),
  \end{multline*}
  with \( M_{f, v_1}^{\chi_1}(m_2 b, h; M_1) \) defined as in~\eqref{eq:def-MTsmooth-chi}, for which we can use Lemma~\ref{lem:compare-MTsmooth-chi-dt}.
  The bound~\eqref{eq:Delta for large M_1, M_2, smooth version} follows after choosing~\( \Omega = X^\frac13 (M_1 M_2)^{-\frac23} \).
\end{proof}

\subsection{Type~II estimates}

The following estimate, the first version of which was obtained in~\cite{Fou85}, concerns convolutions with one component supported inside~$[x^\ee, x^{1/3-\ee}]$.

\begin{lemma}\label{lem:equidistrib-bilin}
For all~$\eta, A>0$, there exist~$\delta, B >0$ such that the following holds. Whenever~$X, R\geq 1$, $(a, h)\in\Z^2$, an interval~$I\subset[X/2, X]$, and two sequences~$(\beta_n)$, $(\gamma_n)$ are given, under the conditions~$1 \leq R, \abs{a}, \abs{h} \leq X^\delta$, and
$$ \abs{\beta_n} \leq \d_A(n), \qquad \abs{\gamma_n}\leq \d_A(n), \qquad \gamma_n\neq 0 \implies n\in[X^\eta, X^{\frac13-\eta}], $$
we have
\begin{equation}
\ssum{n\in I} (\beta \ast \gamma)(n) \Delta_h(an; R) \ll_{A, \eta} \d((a, h)) R^{-1/2} X (\log X)^B. \label{eq:bound-AP-typeII}
\end{equation}
\end{lemma}

\begin{proof}
Recall that~$\Delta_h(an; R) \ll RX^\ee$. In the left-hand side of~\eqref{eq:bound-AP-typeII}, the contribution of those~$n$ such that~$(n, (ah)^\infty)>X^\delta$ is therefore at most
$$ RX^\ee \ssum{n\ll X \\ (n,(ah)^\infty)>X^\delta} 1 \ll RX^{1-\delta+\ee}. $$
Next, we have
$$ \ssum{d|(ah)^\infty \\ d\leq X^\delta} \ssum{n\in I \\ d|n \\ (n/d, ah)=1} (\beta \ast \gamma)(n) \Delta_h(an; R) = \ssum{\lambda_1,\lambda_2 |(ah)^\infty \\ \lambda_1\lambda_2 \leq X^\delta} \ssum{mn\in (\lambda_1\lambda_2)^{-1}I \\ (mn, ah)=1 \\ (m, \lambda_2)=1} \beta_{\lambda_1m} \gamma_{\lambda_2n} \Delta_h(a\lambda_1\lambda_2 mn; R). $$
Finally, we note that there are at most~$\O(X^{1/2+\ee})$ tuples~$(\lambda_1, \lambda_2, m, n)$ with~$\lambda_1\lambda_2mn\in I$ for which the expression~$a\lambda_1\lambda_2mn-h$ is a perfect square, and
$$ \Delta_h(a\lambda_1\lambda_2mn; R) = 2 \sum_{\lambda_3|(h, a\lambda_1\lambda_2)} \ssum{q\leq \sqrt{a\lambda_1\lambda_2mn-h}/\lambda_3 \\ (q, a\lambda_1\lambda_2h/\lambda_3^2)=1} \cu_R(mn\tfrac{a\lambda_1\lambda_2}{\lambda_3}\bar{\tfrac{h}{\lambda_3}}; q) + \O(\1_{a\lambda_1\lambda_2mn-h\text{ is a square}}), $$
where the notation~$\cu_R(n; q)$ is defined in formula~(5.1) of~\cite{D-disp}. Now, for each~$(\lambda_1, \lambda_2, \lambda_3)$, the sum
$$ S(\lambda_1, \lambda_2, \lambda_3) = \ssum{q\leq \sqrt{a\lambda_1\lambda_2mn-h}/\lambda_3 \\ (q, a\lambda_1\lambda_2h/\lambda_3^2)=1} \ssum{mn\in (\lambda_1\lambda_2)^{-1}I \\ (mn, ah)=1 \\ (m, \lambda_2)=1}
\cu_R(mn\tfrac{a\lambda_1\lambda_2}{\lambda_3}\bar{\tfrac{h}{\lambda_3}}; q) $$
is of the same shape as in formula (5.6) of~\cite{D-disp}, with three differences:
\begin{enumerate}
\item the quantity~$\d_A(\lambda_1)\d_A(\lambda_2)$ has to be factored out for the condition~(5.4) of~\cite{D-disp} to hold,
\item the sums over~$m$ and~$n$ must be restricted to dyadic intervals, which is done at the cost of an additionnal factor~$(\log x)^2$,
\item the sums over~$m$,~$n$ and~$q$ are not separated.
\end{enumerate}
The last point can be implemented by a standard argument (see \emph{e.g.} page~720 of~\cite{D-disp}), cutting the~$(m, n)$ sums into intervals of type~$[M, (1+\xi)M]\times[N, (1+\xi)N]$ with~$\xi \asymp R^{-1/2}$. Assuming~$\delta$ is small enough in terms of~$\eta$, we obtain
\begin{align*}
S(\lambda_1, \lambda_2, \lambda_3) {}& \ll \d_A(\lambda_1) \d_A(\lambda_2) (\lambda_1\lambda_2)^{-1} X (\log X)^B (\xi + \xi^{-1} R^{-1}) \\
{}& \ll \d_A(\lambda_1)\d_A(\lambda_2) (\lambda_1\lambda_2)^{-1} R^{-1/2} X (\log X)^B.
\end{align*}
We sum this over~$(\lambda_1, \lambda_2, \lambda_3)$ satisfying
$$ \lambda_1\lambda_2|(ah)^\infty, \quad \lambda_1\lambda_2 \leq x^\delta, \quad \lambda_3 | (h, a\lambda_1\lambda_2). $$
Since~$\sum_{\lambda | (ah)^\infty} \d_{2A}(\lambda)\d(\lambda) \lambda^{-1} \ll_A (\log\log x)^{O_A(1)}$, we obtain
$$ \ssum{n\in I} (\beta \ast \gamma)(n) \Delta_h(an; R) \ll \d((a, h))\big\{R X^{1-\delta/2} + R^{-1/2} X (\log X)^{B+1}\big\}, $$
which yields our claim by reinterpreting~$\delta$ and~$B$.
\end{proof}

\section{The case of rational parameters} \label{section:The case of rational parameters}

Let \( \chi_1, \ldots, \chi_T \) be distinct Dirichlet characters mod~\(D\), and the function~\( f \in \F_D^{\d_\Q}(A) \) be defined by
\begin{equation} \label{eq:definition of g}
  \sum_{n = 1}^\infty \frac{ f(n) }{ n^s } := \prod_{j = 1}^T L(s, \chi_j)^{b_j},
\end{equation}
with~\( b_1, \ldots, b_T \in \Q \), which we write in the form
\[ b_j = r_j + \frac{u_j}{v_j} \quad \text{with} \quad r_j \in \Z \quad \text{and} \quad u_j, v_j \in \N \quad \text{such that} \quad 0 \leq u_j < v_j.\]
For notational convenience we also define
\[ \| r \|_1 := \sum_{ 1 \leq j \leq T } | r_j |, \quad \| v \|_1 := \sum_{ 1 \leq j \leq T } v_j. \]
Our goal is to prove estimate~\eqref{eq:main bound} for the function~\(f\) defined in~\eqref{eq:definition of g}.
In fact, we will prove a result which is slightly more precise in term of uniformity in~$D$ and~$T$.
\begin{proposition} \label{prop:bound-largechars-rational}
  Let~\( A, D, T \geq 1 \) be fixed.
  Then we have, for~\( x \geq 3 \), \( I \subset [ x/2, x ] \) and \( f \in \F_D^{\d_\Q}(A) \) as described above, the following estimate,
  \begin{equation} \label{eq:bound-largechars-rational}
    | \Sigma_f(I; a, h; R) | \leq C \d( (a, h) ) D^\frac52 \frac{ x (\log x)^{ B + \omega(D) } }{ R^\frac12 } \| v \|_1 \quad \text{for} \quad 1 \leq a, |h|, R \leq x^\delta,
  \end{equation}
  where \( \delta > 0 \) is some absolute constant, and where \( B, C > 0 \) are constants which depend only on \(A\)~and~\(T\).
\end{proposition}

The rest of this section is now concerned with proving Proposition~\ref{prop:bound-largechars-rational}.

\subsection{Application of the combinatorial identity}

Denote~$\d_z^\chi(n) := \d_z(n) \chi(n)$, so that
\begin{equation} \label{eq:g as a convolution}
  f(n) = \d_{b_1}^{ \chi_1 } \ast \cdots \ast \d_{b_T}^{ \chi_T }.
\end{equation}
The expression on the left hand side of~\eqref{eq:bound-largechars-rational} now reads
\begin{equation} \label{eq:main expression after writing out the convolution}
  \Sigma_f(I; a, h; R) = \sum_{ m_1 \cdots m_T \in I } \d_{b_1}^{ \chi_1 }( m_1 ) \cdots \d_{b_T}^{ \chi_T }( m_T ) \Delta_h( a m_1 \cdots m_T ; R ).
\end{equation}
By Theorem~\ref{th:combinatorial-identity} with~$K=4$ we can write~\( \d_{b_j}^{ \chi_j }( m_j ) \) as
\begin{equation} \label{eq:decomp-combin}
  \d_{b_j}^{ \chi_j }( m_j ) = \sum_{\ell=1}^4 c_{\ell, j} \mathop{\sum \cdots \sum}_{ \substack{ m_1 \dotsb m_{ k_{\ell, j} } n_1 \dotsb n_{ k'_{\ell, j} } = m_j \\ n_1, \ldots, n_{ k'_{\ell, j} } \leq x^{1/4} }} \chi_j(m_1) \cdots \chi_j( m_{ k_{\ell, j} } ) \d_{-\frac1{v_j}}^{\chi_j}(n_1) \dotsb \d_{-\frac1{v_j}}^{\chi_j}(n_{ k'_{\ell, j} }),
\end{equation}
where $ ( k_{\ell, j} )_{\ell = 1}^4 $~and~$ ( k'_{\ell, j} )_{\ell = 1}^4 $ are two sequences of integers satisfying
\[ 0 \leq k_{ \ell, j } \leq |r_j| + 4, \quad 1 \leq k_{\ell, j}' \leq ( |r_j| + 4 ) v_j, \]
and where \( ( c_{\ell, j} )_{\ell = 1}^4 \) is a set of complex numbers whose moduli are bounded in terms of~\(A\).
We replace each factor~\( \d_{b_j}^{ \chi_j }(m_j) \) in~\eqref{eq:main expression after writing out the convolution} by its decomposition, and after expanding the resulting expression, we end up with a linear combination (whose coefficients are bounded by~$\O_A(1)$) of~$\O_T(1)$ sums of the form
\begin{equation}
\Xi := \sum_{ \substack{ m_1 \cdots m_k n_1 \cdots n_{k'} \in I \\ n_1, \ldots, n_{k'} \leq x^{1/4} } } \sigma_1(m_1) \cdots \sigma_k(m_k) \varrho_1(n_1) \cdots \varrho_{k'}( n_{k'} ) \Delta(a m_1 \cdots m_k n_1 \cdots n_k'; R),\label{eq:def-Xi}
\end{equation}
where each function \( \sigma_i \) is some Dirichlet character mod~\(D\), where each function~\( \varrho_i \) is equal to~\( \d_{ -1 / v_j }^{ \chi_j } \) for some~\(j\), and where \(k\)~and~\(k'\) are integers bounded by
\[ 0 \leq k \leq 4 T + \| r \|_1 \quad \text{and} \quad 1 \leq k' \leq 4 \| v \|_1 + \sum_{1 \leq j \leq T} |r_j| v_j. \]
We consider each sum~$\Xi$ separately.

Out of technical reasons, it will be necessary to use a smooth dyadic decomposition for the variables~\( m_1, \ldots, m_k \).
Let \( u : (0, \infty) \to \R \) be a smooth and compactly supported function, which satisfies
\[ \operatorname{supp} u \subset [1/4, 2] \quad \text{and} \quad \sum_{\ell \in \Z} u\left( \frac\xi{2^\ell} \right) = 1 \quad \text{for all} \quad \xi \in (0, \infty), \]
and define
\[ u_0(\xi) := \sum_{\ell \leq 0} u\left( \frac\xi{ M_\ell } \right) \quad \text{and} \quad u_\ell(\xi) := u\left( \frac\xi{ M_\ell } \right) \quad \text{for} \quad \ell > 0, \]
where we have set
\[ M_\ell := x^{\frac14 + \eta} 2^\ell, \]
with~\( 0 < \eta < \frac1{24} \) an arbitrary, but fixed constant.
For a \(k\)-tuple~\( \ell = (\ell_1, \ldots, \ell_k) \in \N^k \), we then define
\[ \Xi_\ell := \sum_{ \substack{ m_1 \cdots m_k n_1 \cdots n_{k'} \in I \\ n_1, \ldots, n_{k'} \leq x^{1/4} } } u_{\ell_1}(m_1) \sigma_1(m_1) \cdots u_{\ell_k}(m_k) \sigma_k(m_k) \varrho_1(n_1) \cdots \varrho_{k'}( n_{k'} ) \Delta(a m_1 \cdots m_k n_1 \cdots n_k'; R), \]
so that the sum~\(\Xi\) can be split as
\[ \Xi = \sum_{ \ell \in \N^k } \Xi_\ell. \]
Note that this last sum is in fact finite, since \( \Xi_\ell \) becomes empty if the coordinates of~\(\ell\) are large enough, namely if~\( \ell_1, \ldots, \ell_k \gg \log x \).
We will now estimate the sums~\( \Xi_\ell \) in different ways, depending on the sizes of the supports of the variables~\( m_i \).

\subsection{Case I}\label{sec:case-i}

First assume that \(\ell\) has at least one coordinate, say~\( \ell_1 \), satisfying~\( M_{\ell_1} \geq x^{\frac13 + \eta} \).
Let~\( m_0 := m_2 \dotsb m_k n_1 \dotsb n_{k'} \).
Denoting~\( \sigma_1 = \chi_j \) for some~\(j\), we can use Lemma~\ref{lem:Delta for large M} with~\( X = ax \), \( b = a m_0 \) and~\( M = M_{\ell_1} \) to get
\begin{align*}
  \sum_{ m_1: \,\, m_0 m_1 \in I } u_{\ell_1}(m_1) \sigma_1(m_1) \Delta_h(a m_0 m_1; R) &\ll_{\varepsilon, A} x^\varepsilon \left( D a^\frac13 x^\frac13 + ( a m_0, h D^\infty ) \frac{M_{\ell_1}}{ x^\frac12 } + D^\frac12 R^\frac32 \right).
\end{align*}
This leads to
\begin{equation} \label{eqn: bound for Case I}
  \Xi_\ell \ll_\varepsilon x^\varepsilon \left( D a^\frac13 x^{1 - \eta} + ( a, h D^\infty ) (\log x)^{ \omega(D) } x^\frac12 + x^{ \frac23 - \eta} D^\frac12 R^\frac32 \right),
\end{equation}
where we have made use of the fact that
\[ \sum_{ m_0 \leq \frac x{ M_{\ell_1} } } (m_0, h D^\infty) \leq \sum_{ \substack{ D^\ast \mid D^\infty \\ D^\ast \leq x } } D^\ast \sum_{ m_0 \leq \frac x{ D^\ast M_{\ell_1} } } (m_0, h) \ll_\varepsilon h^\varepsilon \frac x{M_{\ell_1}} \sum_{ \substack{ D^\ast \mid D^\infty \\ D^\ast \leq x } } 1 \ll_\varepsilon \frac{ ( \log x )^{ \omega(D) } x^{1 + \varepsilon} }{ M_{\ell_1} }. \]

\subsection{Case II}\label{sec:case-ii}

Next assume that \(\ell\) has at least two non-zero coordinates, say \( \ell_1 \geq \ell_2 \geq 1\).
We can also assume that~\( x^{\frac14 + \eta} \ll M_{\ell_1}, M_{\ell_2} \ll x^{\frac13 + \eta} \), since the case of larger~\( M_{\ell_1} \) and~\(M_{\ell_2}\) is already treated above.
Let \( m_0 := m_3 \dotsb m_k n_1 \dotsb n_{k'} \). We use Lemma~\ref{lem:Delta for large M_1, M_2} with~\( X = ax \) and~\( b = a m_0 \), which gives
\begin{multline*}
  \sum_{ \substack{ m_1, m_2: \\ m_0 m_1 m_2 \in I } } u_{\ell_1}(m_1) \sigma_1(m_1) u_{\ell_1}(m_1) \sigma_2(m_2) \Delta_h(a m_0 m_1 m_2; R) \\
    \ll_{\varepsilon, A} ( a m_0, D^\infty ) x^\varepsilon \left( D^\frac52 (ax M_{\ell_1} M_{\ell_2})^\frac13 + (h, a m_0) R^\frac32 \frac{ M_{\ell_1} M_{\ell_2} }{ a^\frac12 x^\frac12 } \right),
\end{multline*}
so that altogether we are led to
\begin{equation} \label{eqn: bound for Case II}
  \Xi_\ell \ll_{\varepsilon, A} (a, h) (a, D^\infty) (\log x)^{ \omega(D) } x^\varepsilon \left( D^\frac52 a^\frac13 x^{ 1 - \frac 43 \eta } + (a, h) R^\frac32 x^\frac12 \right).
\end{equation}

\subsection{Case III}\label{sec:case-iii}

Finally, we need to consider the case, where \(\ell\) has at most one non-zero coordinate, say~\( \ell_1 \), for which we have~\( M_{\ell_1} \ll x^{\frac13 + \eta} \).
We split the sum~\( \Xi_\ell \) into two parts,
\[ \Xi_j =: \Xi_\ell^{ (1) } + \Xi_\ell^{ (2) }, \]
according to whether~\( n_1 \cdots n_{k'} > x^\eta \) or~\( n_1 \cdots n_{k'} \leq x^\eta \).

We look first at~\( \Xi_\ell^{ (1) } \).
We split this sum according to the value of
$$ \mu = \min\{1\leq \mu' \leq k':\ n_1 \dotsb n_{\mu'} > x^\eta\}, $$
and write accordingly
$$ \Xi_\ell^{ (1) } =: \sum_{\mu=1}^{k'} \Xi_\ell^{ (1) }(\mu). $$
After defining
\begin{align*}
  \beta_m &:= \ssum{ m_1 \dotsb m_k n_{\mu+1} \dotsb n_{k'} = m \\ n_{\mu+1}, \dotsc, n_{k'} \leq x^{1/4} } u_{\ell_1}(m_1) \sigma_1(m_1) \cdots u_0(m_k) \sigma_k(m_k) \varrho_{\mu + 1} (n_{\mu+1} ) \dotsb \varrho_{k'}( n_{k'} ),
  \intertext{and}
  \gamma_n &:= \ssum{ n_1 \dotsb n_{\mu} = n \\ n_1 \dotsb n_{\mu-1} \leq x^\eta, \,\, n_1, \dotsc, n_{\mu} \leq x^{1/4} } \varrho_1(n_1) \cdots \varrho_\mu( n_\mu ),
\end{align*}
and renaming~$n\gets n_1 \dotsb n_{\mu}$ and~$m \gets m_1 \dotsb m_k n_{\mu+1} \dotsb n_{k'}$, we can write~\( \Xi_\ell^{ (1) }( \mu ) \) as
$$ \Xi_\ell^{ (1) }(\mu) = \ssum{m, n :\ mn\in I \\ x^\eta < n \leq x^{1/4+\eta}} \beta_m \gamma_n \Delta_h(amn;R). $$
Note that~$\gamma_n=0$ if~$n>x^{1/4+\eta}$.
Moreover, we can bound the quantities \( \beta_m \)~and~\( \gamma_n \) by
$$ \abs{\beta_m} \leq \d_{2 \| r \|_1 + 8 T }(m), \qquad \abs{\gamma_n} \leq \d_{ \| r \|_1 + 4T }(n). $$
Hence we can apply Lemma~\ref{lem:equidistrib-bilin} with~$ A \gets 2 \| r \|_1 + 8 T $, and we see that
$$ \Xi_\ell^{ (1) }(\mu) \ll \d((a, h)) R^{-1/2} x (\log x)^{B_1} \quad \text{for} \quad 1 \leq a, |h|, R \leq x^{\delta_1}, $$
where \( \delta_1, B_1 > 0 \) are certain constants which depend solely on \(\eta\)~and~\(A\).
Summing over~$\mu$, we deduce
\begin{equation} \label{eq:bound-typeII-rational}
  \Xi_\ell^{ (1) } \ll_A \d((a, h)) R^{-1/2} x (\log x)^{B_1} \|v\|_1 \quad \text{for} \quad 1 \leq a, |h|, R \leq x^{\delta_1}.
\end{equation}

The other sum~\( \Xi_\ell^{ (2) } \) can be estimated similarly -- the role of the variables~\( n_1, \ldots, n_{k'} \) is now played by the variables~\( m_2, \ldots, m_k \).
Eventually, we get
\begin{equation} \label{eq:bound-typeII-rational-2}
  \Xi_\ell^{ (2) } \ll_A \d((a, h)) R^{-1/2} x (\log x)^{B_2} \quad \text{for} \quad 1 \leq a, |h|, R \leq x^{\delta_2},
\end{equation}
where \( \delta_2, B_2 > 0 \) are certain constants which again depend solely on \(\eta\)~and~\(A\).

\subsection{Conclusion}

Grouping the different bounds \eqref{eqn: bound for Case I}--\eqref{eq:bound-typeII-rational-2}, setting \( B := \max(B_1, B_2) \) and choosing~\( \delta > 0 \) small enough, we get
\[ \Xi \ll \d( (a, h) ) D^\frac52 R^{ -\frac12 } x (\log x)^{ B + \omega(D) } \| v \|_1 \quad \text{for} \quad 1 \leq a, |h|, R \leq x^\delta, \]
with the implicit constant depending only on~\( A\) and~\(T\).
This finally proves Proposition~\ref{prop:bound-largechars-rational}.

\section{Interpolation to complex parameters} \label{section: interpolation to complex parameters}

Let~\( r_1, \ldots, r_{ \varphi(D) } \) be the residues mod~\(D\) which are relatively prime to~\(D\).
Any~\( f \in \mathcal F_D^\omega(A) \) is given by
\begin{align} \label{eq:functions in F''_D(A)}
  \sum_{n = 1}^\infty \frac{ f(n) }{n^s} = \prod_{j = 1}^{ \varphi(D) } \prod_{ p \equiv r_j \bmod D } \left( 1 + \frac{ z_j }{ p^s - 1 } \right),
\end{align}
for~\( \bz = (z_1, \ldots, z_{ \varphi(D) }) \in\C^{\vphi(D)} \), with~\(\abs{z_j}\leq A\).
After setting
\begin{equation} \label{eq: definition of omega_r}
  \omega_r(n) := \#\left\{ p \text{ prime} \,:\, p \mid n, \,\, p \equiv r \bmod D \right\},
\end{equation}
we can also write
\[ f(n) = \sum_{ n_1 \cdots n_{ \varphi(D) } = n } \prod_{ j = 1 }^{ \varphi(D) } {z_j}^{ \omega_{r_j}(n_j) }. \]
Our aim here is to show that the bound~\eqref{eq:main bound} holds for~\( \Sigma_f(I; a, h; R) \), for all~\( f \in \mathcal F_D^\omega(A) \).
By Lemma~\ref{lemma: F' or F'' implies F} this will imply Proposition~\ref{prop:bound-largechars}.

Let~\( \chi_1, \ldots, \chi_{ \varphi(D) } \) be the Dirichlet characters mod~\(D\), let~$Q$ be the unitary matrix
\[ Q := \frac1{ \sqrt{ \varphi(D) } } \begin{pmatrix}
     \chi_1( r_1 ) & \chi_2( r_1 ) & \cdots & \chi_{ \varphi(D) }( r_1 ) \\
     \chi_1( r_2 ) & \chi_2( r_2 ) & \cdots & \chi_{ \varphi(D) }( r_2 ) \\
     \vdots & \vdots & \ddots & \vdots \\
     \chi_1( r_{ \varphi(D) } ) & \chi_2( r_{ \varphi(D) } ) & \cdots & \chi_{ \varphi(D) }( r_{ \varphi(D) } )
   \end{pmatrix}, \]
and let \( M_Q : \C^{ \varphi(D) } \to \C^{ \varphi(D) } \) be the bijective linear map associated to~$Q$.

Let~\( K \geq 1 \).
We define~\( \F_D^\omega(A, K) \) to be the set of functions~\( f \in \mathcal F_D^\omega(A) \) of the same form as in~\eqref{eq:functions in F''_D(A)}, but with the additional property that the parameters~\( \bz \) are given by
\[ \bz = M_Q( \bb ) \]
for a tuple of rational numbers \( \bb = (b_1, \ldots, b_{ \varphi(D) }) \in \Q^{\vphi(D)} \) satisfying
\[ | b_j | \leq A \quad \text{and} \quad b_j = \frac{u_j}{v_j} \quad \text{with} \quad u_j, v_j \in \Z \quad \text{and} \quad | v_j | \leq K, \]
for all~\( j = 1, \ldots, \varphi(D) \).
By Proposition~\ref{prop:bound-largechars-rational}, Lemma~\ref{lemma:replacing f by g} and Lemma~\ref{lemma: existence of g_1 and g_2}, we deduce that the bound~\eqref{eq:main bound} holds for all~\( f \in \F_D^\omega(A, K) \) in the following form.

\begin{proposition} \label{lemma:bound for f in tilde F}
  Let~\( A, D \geq 1 \) be fixed.
  For~\( K \geq 1 \), \( x \geq 3 \), \( I \subset [ x/2, x ] \) and~\( f \in \F_D^\omega(A, K) \), we have
  \begin{equation}
    | \Sigma_f(I; a, h; R) | \leq C K \d( (a, h) ) \frac{ x (\log x)^B }{ R^\frac12 } \quad \text{for} \quad 1 \leq a, |h|, R \leq x^\delta,
  \end{equation}
  where \( \delta > 0 \) is some absolute constant, and where \( B, C > 0 \) are constants which depend only on~\(A\) and~\(D\).
\end{proposition}

Our goal is to interpolate this result to all functions in~\( \mathcal F_D^\omega(A) \).
Let~\( f \in \mathcal F_D^\omega(A) \) be fixed, with~\( \bz \) as in~\eqref{eq:functions in F''_D(A)}.
For~\( L \in [1, \infty] \), we define two polynomials in the variables~\( \bZ = (Z_1, \ldots, Z_{ \varphi(D) }) \) as follows,
\[ P_L( \bZ ) := \sum_{ \substack{ n \in I \\ \forall j, \omega_{r_j}(n) \leq L } } \sum_{ n_1 \cdots n_{ \varphi(D) } = n } \prod_{ j = 1 }^{ \varphi(D) } {Z_j}^{ \omega_{r_j}(n_j) } \Delta_h(an; R), \qquad \tilde P_L( \bZ ) := P_L( M_Q(\bZ) ). \]
By definition, both these polynomials have degree at most~\(L\) in each variable.
Furthermore, let
\[ \bb := M_Q^{-1}( \bZ ), \]
and note that~\( \|\bb\|_\infty \leq D^\frac12 A \).
Using this notation, we can now write the sum~\( \Sigma_f(I; a, h; R) \) simply as
\[ \sum_{n \in I} f(n) \Delta_h(an; R) = \tilde P_\infty( \bb ). \]

In order to have better control over the degree of~\( \tilde P_\infty( \bZ ) \), we cut off all the terms of degree larger than some fixed real number~\( L \geq 1 \).
For a tuple~$\bzt = (\zeta_1, \dotsc, \zeta_{\vphi(D)})$ satisfying~\( | \zeta_j | \leq A D^\frac12 \) and any real number~\( E \geq 1 \), this leads to an error term of the following form,
\begin{align*}
  | \tilde P_\infty( \bzt ) - \tilde P_L( \bzt ) | &\leq \sum_{ \substack{ n \in I \\ \omega(n) > L } } \d_D(n) ( A D )^{ \omega(n) } \left| \Delta_h(an; R) \right| \\
    &\leq E^{-L} \sum_{ n \leq x } \d_D(n) (ADE)^{ \omega(n) } \left| \Delta_h(an; R) \right| \\
    &\leq E^{-L} \left( \sum_{ n \leq x } \d_{AD^2E}(n)^2 \right)^\frac12 \left( \sum_{ n \leq x } \left| \Delta_h(an; R) \right|^2 \right)^\frac12.
\end{align*}
The different factors can be estimated via~\cite[Theorem~II.6.1]{Ten15}, and Lemma~\ref{lem:rough-bound-sum-d-dt}, and we get
\begin{align}
  | \tilde P_\infty( \bzt ) - \tilde P_L( \bzt ) | & \ll E^{-L} \left( x (\log x)^{ (ADE)^4 - 1 } \right)^\frac12 \left( x (\log x)^4 \d( (a, h) )^2 \right)^\frac12 \nonumber \\
    &\ll E^{-L} \d( (a, h) ) x (\log x)^{ \frac{ (ADE)^4 }2 + 2 }, \label{eq:bound after cutting off}
\end{align}
where the implicit constants depend at most on~\(A\),~\(E\) and~\(D\).

Next, we set
\[ \beta_\ell := \frac{ 2 (\ell + 1) \lfloor A D^\frac12 \rfloor }{ L + 1 } - \lfloor A D^\frac12 \rfloor \quad \text{for} \quad \ell = 0, \ldots, L. \]
Obviously, all these numbers are bounded by~\( | \beta_\ell | \leq A D^\frac12 \), and are rational numbers with denominators not larger than~\( L + 1 \).
Furthermore, we have the bound
\[ | \beta_{\ell_1} - \beta_{ \ell_2 } | \geq \frac{ A D^\frac12 }{ 2L } | \ell_1 - \ell_2 | \quad \text{for} \quad \ell_1 \neq \ell_2. \]
For any tuple~\( \bell = ( \ell_1, \ldots, \ell_{ \varphi(D) } ) \in \{ 0, \ldots, L \}^{ \varphi(D) } \), denote~$\bbt_\bell = ( \beta_{\ell_1}, \ldots, \beta_{ \ell_{ \varphi(D) } } )$. The value~\( \tilde P_\infty( \bbt_\bell ) \) can be interpreted as an instance of the sum~\( \Sigma_{\tilde f}(I; a, h; R) \) for an appropriate function~\( \tilde f \in \F_D^\omega(AD, L + 1) \),
\[ \tilde P_\infty( \bbt_\bell ) = \Sigma_{\tilde f}(I; a, h; R). \]
Hence, by Proposition~\ref{lemma:bound for f in tilde F} and the estimate in~\eqref{eq:bound after cutting off} we can deduce
\begin{equation} \label{eq: known bound for P_L}
  \tilde P_L( \bbt_\bell ) \ll_{A, D} \d( (a, h) ) x ( \log x )^{ \frac{ ( AD E )^4 }2 + B } \left( \frac L{ R^\frac12 } + \frac1{ {E}^L } \right),
\end{equation}
uniformly for~\( 1 \leq a, |h|, R \leq x^\delta \).

By Lagrange interpolation, we bring~\( \tilde P_L( \bb ) \) into the following shape,
\[ \tilde P_L( \bb ) = \sum_{ \bell \in \{ 0, \ldots, L \}^{ \varphi(D) } } \tilde P_L( \bbt_\bell ) \prod_{ j = 1 }^{ \varphi(D) } \prod_{ \substack{ 0 \leq i \leq L \\ i \neq \ell_j } } \frac{ b_j - \beta_i }{ \beta_{\ell_j} - \beta_i }, \]
which is allowed since the Vandermonde determinant associated to~$(\beta_\ell)$ does not vanish. We can now estimate~\( \tilde P_L( \bb ) \) via the already known bound~\eqref{eq: known bound for P_L} for the expressions~\( \tilde P_L( \bbt_\bell ) \).
Namely, we have
\begin{align*}
  | \tilde{P_L} ( \bb ) | &\leq \sum_{ \bell \in \{ 0, \ldots, L \}^{ \varphi(D) } } \left| \tilde P_L( \bbt_\bell ) \right| \prod_{ j = 1 }^{ \varphi(D) } \prod_{ \substack{ 0 \leq i \leq L \\ i \neq \ell_j } } \frac{ | b_j - \beta_i | }{ | \beta_{\ell_j} - \beta_{\ell_i} | } \\
    &\ll \d( (a, h) ) x ( \log x )^{ \frac{ ( AD E )^4 }2 + B } \left( \frac L{ R^\frac12 } + \frac1{ E^L } \right) (4L)^{ L \varphi(D) } \sum_{ \ell \in \{ 0, \ldots, L \}^{ \varphi(D) } } \prod_{ j = 1 }^{ \varphi(D) } \prod_{ \substack{ 0 \leq i \leq L \\ i \neq \ell_j } } \frac1{ | \ell_j - i | } \\
    &\ll \d( (a, h) ) x ( \log x )^{ \frac{ ( AD E )^4 }2 + B } \left( \frac L{ R^\frac12 } + \frac1{ E^L } \right) \frac{ (8L)^{ L \varphi(D) } }{ (L!)^{ \varphi(D) } },
\end{align*}
which after using Stirling's approximation for the Gamma function simplifies to
\[ | \tilde{P_L}( \bb ) | \ll \d( (a, h) ) x ( \log x )^{ \frac{ ( AD E )^4 }2 + B } \left( \frac1{ R^\frac12 } + \frac1{ E^L } \right) (4e)^{ 2 D L }, \]
with the implicit constant depending at most on \(A\), $E$ and~\(D\).

After adding all the terms we had cut off earlier, we are finally led to
\[ \Sigma_f(I; a, h; R) \ll_{A, D, E} \d( (a, h) ) x ( \log x )^{ \frac{ ( AD E )^4 }2 + B } \left( \frac1{ R^\frac12 } + \frac1{ E^L } \right) (4e)^{ 2 D L }. \]
With the choices
\[ L := \frac{ \log R }{12 D \log(4e) } \quad \text{and} \quad E := (4e)^{6D}, \]
and after reinterpreting the constant~\( B \), we get
\[ \Sigma_f(I; a, h; R) \ll_{A, D} \d( (a, h) ) \frac{ x ( \log x )^B }{ R^\frac13 } \quad \text{for} \quad 1 \leq a, |h|, R \leq x^\delta, \]
which is exactly the statement we wanted to prove.

\section{Proof of~Theorem~\ref{th:main} using Linnik's identity} \label{section: proof using Linnik's identity}

We now sketch how Theorem~\ref{th:main} can alternatively be proven using Theorem~\ref{th:combinatorial-identity-linnik}.
The details of the computations being very similar, we will restrict to discussing the main differences in the arguments.

As mentioned above, it is enough to consider the case~\( f \in \mathcal F_D^\d(A) \), or in other words we can assume that~\( f = \d_{b_1}^{\chi_1} \ast \cdots \ast \d_{b_T}^{\chi_T} \), where \( \chi_1, \ldots, \chi_T \) are distinct Dirichlet characters mod~\(D\), and where \( b_1, \ldots, b_T \) are complex numbers whose moduli are bounded by~\(A\).
The sum in consideration is then given by
\[ \Sigma_f(I; a, h; R) = \sum_{ m_1 \cdots m_T \in I } \d_{b_1}^{ \chi_1 }( m_1 ) \cdots \d_{b_T}^{ \chi_T }( m_T ) \Delta_h( a m_1 \cdots m_T ; R ). \]
Here we replace each~\( \tau_{b_j}^{\chi_j}(m_j) \) by its decomposition as given in Theorem~\ref{th:combinatorial-identity-linnik} with~\( K = 4 \), and after expanding the resulting expression, we end up with a linear combination of sums of the form
\[ \Xi := \sum_{ \substack{ m_1 \cdots m_k n_1 \cdots n_T \in I \\ P^+(n_1 \cdots n_T) \leq x^{1/4} } } \sigma_1(m_1) \cdots \sigma_k(m_k) \rho_1(n_1) \cdots \rho_T(n_T) \Delta( a m_1 \cdots m_k n_1 \cdots n_T; R ), \]
where each function~\( \sigma_j \) is some Dirichlet character mod~\(D\), where each function~\( \rho_j \) is equal to~\( \tau_{b_j - \ell}^{\chi_j} \) for some \(j\) and~\( \ell \in [0, 3] \), and where~\( k \leq 3 T \).
We consider each sum~\( \Xi \) separately.

To each factor~\( \rho_j \) in the sum~\( \Xi \) we apply Lemma~\ref{lem:friable-well-fact} with~\( y = x^{1/4} \) and~\( w = x^\eta \) for some arbitrary, but fixed~\( \eta \in (0, 1/24) \).
By compacity, it follows that for each~\( j = 1, \ldots, T \) there exist arithmetic functions~\( \alpha_j \) and~\( \beta_j \), such that the sum~\( \Xi \) can be written as
\[ \Xi = \sum_{j = 1}^T \Xi_j^{ (1) } + \sum_{j = 1}^T \Xi_j^{ (2) } + \Xi^{ (3) }, \]
with
\begin{align*}
  \Xi_j^{ (1) } &:= \sum_{ \substack{ m n_1 \cdots n_T \in I \\ P^+( n_1 \cdots n_T ) \leq x^{1/4} \\ n_1, \ldots, n_{j - 1} \leq x^\eta, \,\, n_j > x^{1/\eta} \\ \exists p^k \| n_j, \,\, p^k > x^{1/4} } } ( \sigma_1 \ast \cdots \ast \sigma_k )(m) \rho_1(n_1) \cdots \rho_T(n_T) \Delta( a m n_1 \cdots n_T; R ), \\
  \Xi_j^{ (2) } &:= \sum_{ \substack{ m n_1 \cdots n_{j_1} n_j' n_j'' n_{j + 1} \cdots n_T \in I \\ P^+( n_1 \cdots n_{j_1} n_j' n_j'' n_{j + 1} \cdots n_T ) \leq x^{1/4} \\ n_1, \ldots, n_{j-1} \leq x^\eta, \,\, x^\eta < n_j' \leq x^{1/4 + \eta} } } \!\!\! ( \sigma_1 \ast \cdots \ast \sigma_k )(m) \Bigg( \prod_{ \substack{1 \leq k \leq T \\ k \neq j } } \rho_k(n_k) \Bigg) \alpha_j(n_j') \beta_j(n_j'') \Delta( a m n_1 \cdots n_T ; R ), \\
  \Xi^{ (3) } &:= \sum_{ \substack{ m n_1 \cdots n_T \in I \\ n_1, \ldots, n_T \leq x^\eta } } ( \sigma_1 \ast \cdots \ast \sigma_k )(m) \rho_1(n_1) \cdots \rho_T(n_T) \Delta( a m n_1 \cdots n_T; R ).
\end{align*}

The sums~\( \Xi_j^{ (1) } \) can be bound trivially.
Indeed, we note that if a prime power~\( p^k > y \) divides~\(n\), then since~\( P^+(n_j) \leq y \) we must have~\( k \geq 2 \).
Hence
\begin{align*}
  \Xi_j^{ (1) } &\leq \sum_{ \substack{ n \in I \\ \exists p^k \mid n: \,\, p^k > x^{1/4}, \,\, k \geq 2 } } \d_{ (A + 6) T }(n) | \Delta(an; R) | \\
    &\ll_\ee \left( \sum_{n \in I} | \Delta(an; R) |^2 \right)^\frac12 \Bigg( \sum_{ \substack{ n \in I \\ \exists p^k \mid n: \,\, p^k > x^{1/4}, \,\, k \geq 2 } } 1 \Bigg)^\frac12 \\
    &\ll_{A, T} x^{1-1/17} \tau( (a, h) ),
\end{align*}
which is an acceptable error term.

Concerning the sums~\( \Xi_j^{ (2) } \), we can bound them following the arguments of Case~III, Section~\ref{sec:case-iii}, since we have a variable localized in~$[x^\eta, x^{1/4+\eta}]$, and since~$1/4+\eta < 1/3$.
The remaining sum~\( \Xi^{ (3) } \), which is analogous to~\eqref{eq:def-Xi}, can be estimated for all sufficiently small~$\eta>0$ by the arguments of Sections~\ref{sec:case-i}, \ref{sec:case-ii} and \ref{sec:case-iii}, according to the size of the involved variables.
As a result, we get for these sums the estimate
\[ \Xi_j^{ (2) }, \Xi^{ (3) } \ll_{A, T} \tau( (a, h) ) \frac{ x (\log x)^{ \mathcal O(1) } }{ R^\frac12 }. \]
Together with the bound for~\( \Xi_j^{ (1) } \), this eventually proves Theorem~\ref{th:main}.

\section{Proof of Theorems~\ref{th:main}, \ref{cor:piltz}, \ref{cor:norms of ideals} and~\ref{cor:titchmarsh}} \label{section:proof of main theorem}

In this section we want to deduce Theorem~\ref{th:main} from Proposition~\ref{prop:bound-largechars}, and afterwards apply this result to the problems mentioned in the introduction.
Before doing so, we first need to prove an auxiliary result, which is concerned with bounds on average for functions in~\( \mathcal F_D(A) \) twisted by a Dirichlet character.

\begin{lemma} \label{lemma: bound for the average of psi(n) f(n)}
  Let \( f \in \mathcal F_D(A) \) and let~\( B \geq 1 \).
  Then there exists a constant~\( c > 0 \), such that, for all Dirichlet characters~\( \chi \) mod~\(q\) satisfying~\( \cond(\chi) \nmid D \) and~\( q \leq (\log x)^B \), we have
  \begin{equation} \label{eq: bound for chi f}
    \sum_{n \leq x} \chi(n) f(n) \ll x e^{ -c \sqrt{ \log x } }.
  \end{equation}
  Both the constant \(c\) and the implicit constant depend at most on~\(A\), \(B\) and~\(D\).
\end{lemma}
\begin{proof}
  Let \( F_\chi(s) \) be the Dirichlet series associated to the function~\( \chi(n) f(n) \).
  By Lemma~\ref{lemma: Dirichlet series of f} we know that~\( F_\chi(s) \) can be written as
  \[ F_\chi(s) = H_\chi(s) \prod_{ \psi \bmod D } L(s, \chi \psi)^{b_\psi} \quad \text{for} \quad \Re(s) > 1, \]
  where~\( H_\chi(s) \) is a holomorphic function in~\( \Re(s) \geq \frac12 + \varepsilon \), bounded in terms of~$A, D$ only.

  Due to the assumption~\( \cond(\chi) \nmid D \) we know that none of the characters~\( \chi \psi \) is principal, which means that none of the \(L\)-functions~\( L(s, \chi \psi) \) has a pole at~\( s = 1 \).
  It follows from Siegel's theorem that for any~\( \delta > 0 \) there exists a constant~\( c(\delta) \) such that all~\( L(s, \chi \psi) \) are zero-free in the region defined by the condition~\( \Re(s) > 1 - \gamma( \Im(s) ) \), where
  \begin{equation} \label{eq: zero-free region}
    \gamma(t) := \min\left\{ \frac{ c(\delta) }{ \log( qD ( | t | + 2 ) ) }, \frac{ c(\delta) }{ (qD)^\delta } \right\}.
  \end{equation}
  Using this zero-free region, the bound~\eqref{eq: bound for chi f} follows using a standard contour integration argument; see \textit{e.g.}~\cite[Section~11.3]{MV}.
\end{proof}

We now proceed to prove Theorem~\ref{th:main}.
We set~\( R = (\log x)^L \) where \( L \geq 1 \) is some constant which depends only on~\(A\), \(B\) and~\(D\), and which we will determine at the very end.
Note that in any case we can assume \(x\) to be large enough so that~\( D \leq R \) is satisfied.

We start by splitting the sum~\( D_f(x; a, h) \) into two parts as follows,
\[ D_f(x; a, h) = D_f( \sqrt{x}; a, h) + \sum_{ \sqrt{x} < n \leq x } f(n) \d(an - h). \]
While the first sum can be estimated by trivial means, we can use Proposition~\ref{prop:bound-largechars} to evaluate the second (after first dividing the range of summation into dyadic intervals).
This eventually shows that there exists an absolute constant~\( \delta > 0 \), and a constant~\(B\) depending only on~\(A\) and~\(D\), such that, for all~\( 1 \leq a, |h| \leq x^\delta \),
\[ D_f(x; a, h) = M_f(x; a, h) + \O\left( \d( (a, h) ) \frac{ x (\log x)^B }{ R^{1/3} } \right), \]
with
\[ \tilde M_f(x; a, h) := \sum_{ |h|/a < n \leq x } f(n) \dt_h(an; R). \]
It remains to evaluate this last sum.

After expanding~\( \dt_h(an; R) \), it can be written as
\[ \tilde M_f(x; a, h) = 2 \sum_{ \substack{ q \leq \sqrt{ ax } } } \frac1{ \varphi\left( \frac q{ (h, q) } \right) } \sum_{ \substack{ \chi \bmod \frac q{ (h, q) } \\ \cond \chi \leq R } } \overline\chi\left( \tfrac h{ (h, q) } \right) \sum_{ \substack{ \frac{q^2}a \leq n \leq x \\ (an, q) = (h, q) } } f(n) \chi\left( \tfrac{an}{ (h, q) } \right) + \O\left( x^{\delta + \varepsilon} \right). \]
We now split the remaining sum into two parts, denoted by \( \tilde M_f^{ (1) }(x; a, h) \)~and~\( \tilde M_f^{ (2) }(x; a, h) \), depending on whether~\( \cond(\chi) \mid D \) or not.
A simple reordering of the sums shows that the first part is equal to~\( M_f(x; a, h) \) as given in Theorem~\ref{th:main}.
The second part can be written as
\[ M_f^{ (2) }(x; a, h) = 2 \sum_{ t \mid (a, h) } \sum_{ \substack{ u \mid \frac h{t} \\ (u, a/t) = 1 } } \sum_{ q \leq \frac{ \sqrt{ ax } }{ t u } } \frac1{ \varphi(q) } \sum_{ \substack{ \chi \bmod q \\ \cond \chi \leq R \\ \cond(\chi) \nmid D } } \overline\chi\left( \tfrac h{ t u } \right) \chi\left( \tfrac a{ t } \right) \left( S_{f, \chi}(x, u) - S_{f, \chi}\left( \tfrac{ t u^2 q^2 }a, u \right) \right), \]
with~\( S_{f, \chi}(x, u) \) given by
\[ S_{f, \chi}(x, u) := \sum_{ n \leq \frac xu } f(u n) \chi(n). \]
This last sum can be estimated via Lemma~\ref{lemma: bound for the average of psi(n) f(n)}, namely we have
\begin{align*}
  S_{f, \chi}(x, u) &= \sum_{ \substack{ u^\ast \mid u^\infty \\ u^\ast \leq \sqrt{x} } } f(u u^\ast) \chi(u^\ast) \sum_{ \substack{ n \leq \frac x{u u^\ast} \\ (n, u) = 1 } } f(n) \chi(n) + \O\left( x^{\frac12 + \varepsilon} \right) \\
    &\ll x e^{ -c \sqrt{\log x} } \sum_{ u^\ast \leq \sqrt{x} } \frac{ \d_A(u u^\ast) }{u u^\ast} + x^{\frac12 + \varepsilon} \\
    &\ll \frac{ \d_A(u) }u x (\log x)^A e^{ -c \sqrt{\log x} },
\end{align*}
for some constant~\( c > 0 \) depending on~\(A\), \(D\) and~\(L\).
Hence
\[ M_f^{ (2) }(x; a, h) \ll \d( (a, h) ) R x (\log x)^{A + 2} e^{ -c \sqrt{\log x} } \ll \d( (a, h) ) x e^{ -\frac c2 \sqrt{\log x} }. \]

Eventually, we get
\[ D_f(x; a, h) = M_f(x; a, h) + \O\left( \d( (a, h) ) x \left( \frac{ (\log x)^B }{ R^{1/3} } + e^{ -\frac c2 \sqrt{\log x} } \right) \right), \]
and Theorem~\ref{th:main} follows with the choice~\( L = 3N + 3B\).

\subsection{Proof of Theorems~\ref{cor:piltz}, \ref{cor:norms of ideals} and~\ref{cor:titchmarsh}}

The applications mentioned in the introduction are essentially all immediate corollaries of~Theorem~\ref{th:main}, except for the fact that it remains to evaluate the main terms. This is a rather tedious task, but can be done using standard techniques from analytic number theory, in particular the Selberg-Delange method, which is for example described in detail in~\cite[Chapter~II.5]{Ten15}.
In order to not further lengthen this article, we only want to indicate very briefly the main steps of the procedure.

In the case of Theorem~\ref{cor:piltz}, the main term takes the form
\[ M_{\d_z}(x; 1, h) = 2 \sum_{ q \leq \sqrt{ x } } \frac1{ \varphi\left( \frac q{ (h, q) } \right) } \sum_{ \substack{ q^2 \leq n \leq x \\ (n, q) = (h, q) } } \d_z(n), \]
which after a few simple transformations can be written as
\begin{equation} \label{eq: main term for d_alpha}
  M_{\d_z}(x; 1, h) = \sum_{ \substack{ u \mid h, \,\, v \mid u^\infty \\ v \leq \sqrt{x} } } \d_z(uv) \sum_{ \substack{ q \leq \sqrt{x} / u \\ \left( q, \frac{vh}u \right) = 1 } } \frac{ D(x; uq, uv) - D(u^2 q^2; uq, uv) }{ \varphi(q) } + \O\left( x^{\frac12 + \varepsilon} \right),
\end{equation}
where
\[ D(y; r, t) := \sum_{ \substack{ n \leq \frac yt \\ (n, r) = 1 } } \d_z(n). \]
This sum has been studied in detail in~\cite[Chapter~II.5]{Ten15}.
In particular, following the proof of~\cite[Theorem~II.5.2]{Ten15}, we see that there exist complex numbers~\( \mu_\ell^z(r, t) \) such that
\[ D(y; r, t) = \frac1{2\pi i} \sum_{\ell = 0}^L \frac{ \mu_\ell^z(r, t) }{ \Gamma(z - \ell) } \frac{ y (\log y)^z }{ (\log y)^{\ell + 1} } + \O\left( \frac{ (\log t)^{L + 1} }t \frac{ y (\log y)^z }{ (\log y)^{L + 2 - \varepsilon} } \right), \]
where
\[ \mu_\ell^z(r, t) := \Delta_s^\ell \left( \frac{ \psi_s^z(r) }{t^s} \frac{ (s - 1)^z \zeta(s)^z }s \right) \quad \text{with} \quad \psi_s^z(r) := \prod_{p \mid r} \left( 1 - \frac1{p^s} \right)^z, \]
and where the differential operator~\( \Delta_s^\ell \) is defined as
\[ \Delta_s^\ell := \frac1{ \ell! } \frac{ \partial^\ell }{ \partial s^\ell } \bigg|_{s = 1}. \]

It therefore remains to evaluate the sums
\[ \sum_{ \substack{ q \leq \frac{ \sqrt{x} }u \\ \left( q, \frac{vh}u \right) = 1 } } \frac{ \Delta_s^\ell \psi_s^z(uq) }{ \varphi(q) } \quad \text{and} \quad \sum_{ \substack{ q \leq \frac{ \sqrt{x} }u \\ (q, vh/u) = 1 } } \frac{ \Delta_s^\ell \psi_s^z(uq) }{ \varphi(q) } \frac{ q^2 ( 2 \log(uq) )^z }{ ( 2 \log(uq) )^{\ell - 1} }. \]
For the first sum this is a standard exercise in using counter integration, the result being
\[ \sum_{ \substack{ q \leq \frac{ \sqrt{x} }u \\ \left( q, \frac{vh}u \right) = 1 } } \frac{ \Delta_s^\ell \psi_s^z(uq) }{ \varphi(q) } = \Delta_s^\ell \underset{w = 0}{ \operatorname{Res} } \left( C_{s, w}^z \frac{ \psi_s^z(u) \rho_w\left( \frac{vh}u \right) }{ \gamma_{s, w}^z(h) } \frac{ x^\frac w2 }{ u^w } \frac{ \zeta(w + 1) }w \right) + \O\left( \frac1{ x^{\frac23 - \varepsilon} } \right), \]
with
\[ C_{s, w}^z := \prod_p \left( 1 + \frac1{ (p - 1) p^{w + 1} } + \frac{ \psi_s^z(p) - 1 }{ (p - 1) p^w } \right), \]
and
\[ \gamma_{s, w}^z(n) := \prod_{ p \mid n } \left( 1 + \frac{ p ( \psi_s^z(p) - 1 ) }{ p^{w + 2} - p^{w + 1} + 1 } \right) \quad \text{and} \quad \rho_w(n) := \prod_{p \mid n} \left( 1 - \frac p{ p^{w + 2} - p^{w + 1} + 1 } \right). \]
An asymptotic formula for the second sum now follows via partial summation.
After putting the resulting formulae back in~\eqref{eq: main term for d_alpha} and completing the sum over~\(v\), this eventually leads to the main term described in Theorem~\ref{cor:piltz}.
In particular, the first coefficient is given by
\begin{equation}
  \label{eq:first-coeff-dz}
  \begin{aligned}
    \lambda_{h, 0}(z) &= \frac1{ \Gamma(z) } \prod_{ (p, h) = 1 } \Bigg( 1 + \frac{ \big( 1 - \frac1p \big)^{z - 1} - 1 }p \Bigg) \\
    &\phantom{ = {} } \cdot \prod_{ p^\ell \parallel h } \Bigg( 1 - \frac1p + \left( 1 - \frac1p \right)^{z + 1} \sum_{j = 1}^{\ell - 1} \frac{ (\ell - j) \d_z(p^j) }{p^j} + \left( 1 - \frac1p \right)^{z - 1} \frac{ \d_z(p^\ell) }{ p^{\ell + 1} } \Bigg).
  \end{aligned}
\end{equation}

For Theorem~\ref{cor:sums of two squares}, we have from~\cite[Proposition~8.4, Theorem~8.6]{Narkiewicz} that the characteristic function~$n\mapsto b_K(n)$ of the set~$\cN_K$ is multiplicative with~$b(p) = 1$ if and only if~$\sum_{\chi\in X(K)} \chi(p) > 0$, where~$X(K)$ is a subgroup of the Dirichlet characters modulo the discriminant~$D = \text{Disc}(K)$ and~$p\nmid D$. The subgroup of residue classes~$a\bmod{D}$ such that~$\sum_{\chi\in X(K)} \chi(a) > 0$, corresponding to the subgroup~$H$ in~\cite[Theorem~8.2]{Narkiewicz}, has density~$1/[K:\Q]$ inside~$(\Z/D\Z)^\times$. Thus we have a factorization
$$ \sum_{n\geq 1} \frac{b_K(n)}{n^s} = \zeta(s)^{1/[K:\Q]} H(s) $$
where~$H$ is holomorphic and bounded in the strip~$\Re(s)\geq \frac 23$. The rest of the argument the follows the path described above. We leave the details to the reader.

In the case~$K = \Q(i)$, the first coefficient is given by~$\beta_{h, 0} = B_0 B(h)$, where
\begin{align*}
  \numberthis\label{eq:first-coeff-bn}
  B_0 &:= \frac1{ \sqrt{2} } \prod_{ p \equiv 3 \bmod 4 } \left( 1 - \frac1{p^2} \right)^{ -\frac12 },
  \intertext{and}
  B(h) &:= \left( 1 + \frac{ \chi_4(h^\ast) }{4 h^\circ} \right) \prod_{ \substack{ p^\ell \parallel h \\ p \equiv 3 \bmod 4 } } \left( 1 - \frac1{p + 1} + \frac{ (-1)^\ell }{ p^\ell (p + 1) } \right) \prod_{p \equiv 3 \bmod 4} \left( 1 + \frac1{p^2} \right),
\end{align*}
with \( h^\circ := (h, 2^\infty) \), \( h^\ast := \frac h{h^\circ} \) and \(\chi_4\) the non-principal character mod~\(4\).

Finally, the proof of Theorem~\ref{cor:titchmarsh} rests upon the fact that
\[ \sum_{ \substack{ |h| < n \leq x \\ \omega(n) = k } } \d(n - h) = \frac1{k!} \frac{\partial^k }{\partial z^k} \Xi_{x, h}(0) \quad \text{with} \quad \Xi_{x, h}(z) := \sum_{ |h| < n \leq x } z^{ \omega(n) } \d(n - h). \]
Since the function~\( n \mapsto z^{ \omega(n) } \) is an element of~\( \mathcal F_1(A) \) for~\( |z| \leq A \), Theorem~\ref{th:main} can again be applied in this case.
After evaluating the arising main term in the same manner as described above, we see that there exist functions~\( \gamma_{h, \ell}(z) \), which are holomorphic in a neighborhood of~\(z\), such that
\[ \Xi_{x, h}(z) = x (\log x)^z \sum_{\ell = 0}^L \frac{ \gamma_{h, \ell}(z) }{ (\log x)^\ell } + \O\left( \frac{ x (\log x)^{ \Re(z) } }{ (\log x)^{L + 1 - \varepsilon} } \right). \]
Now an application of~\cite[Theorem~II.6.3]{Ten15} proves Theorem~\ref{cor:titchmarsh}.

\bibliographystyle{alpha}
\bibliography{divtm-alpha}

\end{document}